\documentclass[a4paper, 11pt, reqno]{amsart}

\usepackage{amssymb}
\usepackage{amsthm}
\usepackage{bm}
\usepackage{latexsym}
\usepackage{amsmath}
\usepackage{mathrsfs}
\usepackage{caption}
\usepackage{amscd}
\usepackage[all,cmtip]{xy}
\usepackage[usenames]{color}
\usepackage{graphicx}
\usepackage{color}
\usepackage{amscd}
\usepackage{float}
\usepackage{graphics}
\usepackage{tikz}
\usepackage{tikz-cd}
\usepackage{comment}
\usepackage{xspace}
\usepackage{mathtools}
\usepackage{subfig}
\usepackage{enumerate}
\usepackage{hyperref}
\usepackage{enumitem}

\usepackage{marginnote}

\newcommand{\appsection}[1]{\let\oldthesection\thesection
\renewcommand{\thesection}{Appendix \oldthesection}
\section{#1}\let\thesection\oldthesection}

\def\qq{{\mathbb{Q}}}

\def\pp{{\mathbb{P}}}
\newcommand{\defi}[1]{\textsf{#1}} 
\newcommand\cal{\mathcal}
\newcommand\bb{\mathbb}

\DeclareMathOperator{\Cl}{Cl}

\DeclareMathOperator{\Pic}{Pic}
\DeclareMathOperator{\Aut}{Aut}

\DeclareMathOperator{\coeff}{coeff}

  \newtheorem{theorem}{Theorem}[section]
  \newtheorem{lemma}[theorem]{Lemma}
  \newtheorem{proposition}[theorem]{Proposition}
  \newtheorem{corollary}[theorem]{Corollary}

  \newtheorem{definition}[theorem]{Definition}
  \newtheorem{example}[theorem]{Example}

  \newtheorem{question}[theorem]{Question}

\newtheorem{remark}[theorem]{Remark}

\theoremstyle{remark}

\numberwithin{equation}{section}

\calclayout

\thanks{During the preparation of this article, L.J. was partially supported by NSF MSPRF grant DMS-2202444.\\ J.M. was partially supported by NSF research grant DMS-2443425.}

\dedicatory{Dedicated to Caucher Birkar.}

\begin{document}

\title[Toricity in families of Fano varieties]
{Toricity in families of Fano varieties}

\author[L.~Ji]{Lena Ji}
\address{Department of Mathematics, University of Illinois Urbana-Champaign, 214 Harker Hall, 1305 W. Green St., Urbana, IL 61801}
\email{lenaji.math@gmail.com}

\author[J.~Moraga]{Joaqu\'in Moraga}
\address{UCLA Mathematics Department, Box 951555, Los Angeles, CA 90095-1555, USA
}
\email{jmoraga@math.ucla.edu}

\begin{abstract} 
Rationality is not a constructible property in families. In this article, we consider stronger notions of rationality and study their behavior in families of Fano varieties. We first show that being toric is a constructible property in families of Fano varieties.
The second main result of this article concerns an intermediate notion that lies between toric and rational varieties, namely \emph{cluster type} varieties. 
A cluster type $\mathbb Q$-factorial Fano variety contains an open dense algebraic torus, but the variety does not need to be endowed with a torus action.
We prove that, in families of $\mathbb Q$-factorial terminal Fano varieties, being of cluster type is a constructible condition. As a consequence, we show that there are finitely many smooth families parametrizing $n$-dimensional smooth cluster type Fano varieties.
\end{abstract}

\maketitle

\setcounter{tocdepth}{1}
\tableofcontents

\section{Introduction} 

Let \(\cal X \to T\) be a family of projective varieties over \(\mathbb C\). The subset
\[
T_{\rm rational} \coloneqq\{ \text{closed } t \in T \mid \mathcal{X}_t \text{ is a rational variety, i.e., \(\mathcal X_t\) is birational to projective space} \} 
\]
is (the set of closed points of) a countable union of locally closed subsets of \(T\)
(see~\cite[Proposition 2.3]{dFF13}).
In smooth families, \(T_{\rm rational}\) is a countable union of closed subsets \cite{NicaiseShinder,KontsevichTschinkel}; however, for families with singular members, these subsets are in general neither open nor closed \cite{dFF13,Totaro16,Perry17}. Furthermore,
\(T_{\rm rational} \subseteq T\) is not in general a constructible subset,
i.e., it is not a finite union of locally closed subsets (in the Zariski topology) \cite{HassettPirutkaTschinkel}.
There are also families where \(T_{\rm rational} = T\), but the generic fiber is irrational as a \(k(T)\)-variety (see, e.g., \cite{Manin66,FJ24}).

From the perspective of the Minimal Model Program (MMP) it is natural to investigate the behavior of rationality in families of Fano varieties.
In the case of singular degenerations of Fano varieties, Birkar, Loginov, and Qu have shown that the dimension and log discrepancies of general fibers control the irrationality of the special fiber~\cite{BirkarLoginov21,BirkarQu24}.

One important class of rational varieties is the class of toric varieties, i.e., those that contain a dense algebraic torus whose action on itself extends to the whole variety.
Our first result is that being toric is a constructible condition in families of (klt) Fano varieties:

\begin{theorem}\label{thm:toric-Fano}
Let $\mathcal{X} \rightarrow T$ be a family of Fano varieties over \(\mathbb C\).
Then, the set 
\[
T_{\rm toric} \coloneqq\{ \text{closed } t \in T \mid \mathcal{X}_t \text{ is a toric variety}
\} 
\]
is a constructible subset of $T$. 
\end{theorem}
In fact, the Fano assumption is not necessary: being toric is a constructible condition in \emph{any} family of complex varieties (Theorem~\ref{thm:toric-constructible-any}).
We note that being toric is neither an open nor a closed condition (Examples~\ref{ex:toric-not-closed} and~\ref{ex:toric-not-open}).

Being toric places strong restrictions on the geometry of a variety.
For instance, smooth toric Fano varieties do not deform \cite{BienBrion} (see also \cite[Theorem 1.4]{dFH11}), so Theorem~\ref{thm:toric-Fano} is automatic in the smooth case.
Thus, we consider an intermediate class: \defi{cluster type varieties} (see Definition~\ref{defn:toric-pair}). This is a class of rational varieties generalizing toric varieties, and was introduced by Enwright, Figueroa, and the second author in~\cite{EFM24}.
For example, smooth toric varieties are covered by tori, up to a closed codimension two subset (see Proposition~\ref{prop:toric-covering}), whereas smooth cluster type Fano varieties contain an open dense torus (see \cite[Theorem 1.2 and Theorem 1.3(4)]{EFM24}).
We prove the following result about cluster type varieties in families of Fano varieties:

\begin{theorem}\label{thm:cluster-Fano}
Let $\mathcal{X}\rightarrow T$ be a family of $\mathbb{Q}$-factorial terminal Fano varieties over \(\mathbb C\).
The set 
\[
T_{\rm cluster}\coloneqq
\{ \text{closed } t\in T \mid \mathcal{X}_t \text{ is a cluster type variety}\} 
\]
is a constructible subset of $T$. 
\end{theorem}

The $\qq$-factorial terminal assumption in the previous theorem is vital for our proof to work (see Remark~\ref{rem:cluster-type-singularities-assumptions}).
Many properties of families of terminal Fano varieties do not hold if we drop the condition
on the singularities (see, e.g.,~\cite[Remark 6.2]{dFH11}).
As a consequence of Theorem~\ref{thm:cluster-Fano} and a result of Koll\'ar--Miyaoka--Mori \cite{KMM92}, we get the following corollary about cluster type smooth Fano varieties.
\begin{corollary}\label{introcor:ct-smooth-Fano}
Let $n$ be a positive integer. 
Then, there are finitely many smooth projective morphisms $f_i\colon \mathcal{X}_i\rightarrow T_i$ parametrizing 
$n$-dimensional smooth cluster type Fano varieties over \(\mathbb C\).
\end{corollary}
By \defi{parametrizing}, we mean here that every $n$-dimensional smooth cluster type Fano appears as the closed fiber of some $f_i$, and, moreover, \emph{every} closed fiber of each $f_i$ is in this class of varieties.
This notion is stronger than boundedness,
but weaker than moduli, since the same variety could appear as a fiber multiple times.
Furthermore, from the proof of Theorem~\ref{thm:cluster-Fano}, it follows that all fibers of each $f_i$ in Corollary~\ref{introcor:ct-smooth-Fano} admit the same combinatorial structure, in the sense that a certain birational transformation to a toric pair can be performed in the whole family (see Lemma~\ref{lem:dlt-mod-cluster}).

On the other hand, the analogous statement for \emph{rational} varieties is false:
\begin{theorem}\label{thm:rat-fam}
    Let \(n\geq 4\) be an integer. Then, there do not exist finitely many smooth projective morphisms parametrizing \(n\)-dimensional smooth rational Fano varieties.
\end{theorem}

In the toric case, the analogous statement to Corollary~\ref{introcor:ct-smooth-Fano} was already known: smooth toric Fano varieties correspond to smooth reflexive polytopes (see, e.g., \cite[Theorem 8.3.4]{CLS11}), so in each dimension \(n\) there are only finitely many up to isomorphism.
Note that smooth cluster type Fanos do deform, so unlike in the toric case,
the parametrizing varieties \(T_i\) in Corollary~\ref{introcor:ct-smooth-Fano} are in general positive-dimensional.
For example, every smooth del Pezzo surface of degree \(\geq 2\) is of cluster type by~\cite[Theorem 2.1]{ALP23} and~\cite[Lemma 1.13]{GHK15}.

\subsection{Outline}
In Section~\ref{sec:preliminaries}, we begin by recalling definitions and preliminary results on singularities of pairs, toric log Calabi--Yau pairs, and cluster type log Calabi--Yau pairs. We prove several lemmas that we will apply in later sections.
Next, in Sections~\ref{sec:toric} and~\ref{sec:cluster-type}, we prove versions of Theorems~\ref{thm:toric-Fano} and~\ref{thm:cluster-Fano} for families of Fano varieties with a log Calabi--Yau pair structure. Then, in Section~\ref{sec:proofs-of-main-theorems}, we use our results from Sections~\ref{sec:toric} and~\ref{sec:cluster-type} to prove the main theorems. Finally, in Section~\ref{sec:examples-and-questions}, we end the article with examples of different behaviors of toricity in families, and we pose some related questions.

\subsection*{Acknowledgements}
We are grateful to J\'anos Koll\'ar and Burt Totaro for suggesting Theorem~\ref{thm:toric-constructible-any}, and to Brendan Hassett for suggesting Theorem~\ref{thm:rat-fam}.
We also thank Louis Esser and Stefano Filipazzi for helpful conversations and comments, and we thank the anonymous referee for their careful reading and helpful comments.
Part of this work was carried out while the authors were visiting the Yau Mathematical Sciences Center at Tsinghua University, and we thank Caucher Birkar and Spring Ma for their hospitality.

\section{Preliminaries}\label{sec:preliminaries}

We work over an algebraically closed field of characteristic \(0\). 
We write $\Sigma^n$ for the sum of the coordinate hyperplanes of $\pp^n$. Throughout, a Fano variety will mean a projective variety \(X\) such that \(-K_X\) is ample and \(X\) has klt singularities.

In this section, we introduce preliminary results on singularities of the minimal model program, toric and cluster type log Calabi--Yau pairs, families of pairs, and dlt modifications.

\subsection{Singularities of the MMP}

\begin{definition}
{\em 
    A \defi{pair} \((X,B)\) consists of a normal quasi-projective variety \(X\) and an effective $\qq$-divisor \(B\) on \(X\) such that \(K_X+B\) is \(\bb Q\)-Cartier. If \(\{B_i \mid i \in I\}\) are the prime components of \(B\), then the \defi{strata} of \((X,B)\) are the irreducible components of the intersections \(\bigcap_{i \in J, J \subset I} B_i\).
    The \defi{index} of a pair \((X,B)\) is the smallest positive integer \(m\) such that \(m(K_X + B)\) is Cartier.
}
\end{definition}

\begin{definition}
{\em 
    Let \((X,B)\) be a pair. Given a projective birational morphism \(\pi\colon Y \to X\) from a normal quasi-projective variety and a prime divisor \(E\) on \(Y\), the \defi{log discrepancy} of \((X, B)\) at \(E\) is \[a_E(X, B) \coloneqq 1 - \coeff_E(\pi^*(K_X + B)).\]
    A pair \((X,B)\) is \defi{terminal} (resp. \defi{canonical}, \defi{Kawamata log terminal (klt)}, \defi{lc canonical (lc)}) if \(a_E(X, B) > 1\) (resp. \(a_E(X, B) \geq 1\), \(a_E(X, B) >0\) and \(\lfloor B \rfloor = 0\), \(a_E(X, B) \geq 0\)) for every exceptional divisor \(E\) over \(X\).
}
\end{definition}

\begin{definition}
{\em
Let \((X,B)\) be a pair and \(f\colon Y \to X\) a birational morphism. The \defi{log pullback of \(B\)} is the (not necessarily effective) \(\bb Q\)-divisor \(B_Y\) defined by \(K_Y+B_Y = f^*(K_X+B)\) and \(f_* B_Y = B\).
}
\end{definition}

\begin{definition}
{\em 
    Let \((X,B)\) be a log canonical pair. A divisor \(E\) over \(X\) is a \defi{log canonical place} (resp. \defi{canonical place}, \defi{non-canonical place}, \defi{terminal place}, \defi{non-terminal place}) if \(a_E(X, B) = 0\) (resp. \(a_E(X,B) = 1\), \(a_E(X,B) < 1\), \(a_E(X,B) > 1\), \(a_E(X,B)\leq 1\)). The \defi{center} of \(E\) on \(X\) is the closure of its image in \(X\), and \(E\) is \defi{exceptional} if its center on \(X\) is not a divisor.
}
\end{definition}

\begin{definition}
{\em
Let \(\phi\colon Y \dashrightarrow X\) be a rational map of normal projective varieties. We say \(\phi\) \defi{contracts} a Weil divisor \(E \subset Y\) if the center of \(E\) on \(X\) is not a divisor. We say \(\phi\) \defi{extracts} a Weil divisor \(F \subset X\) if \(\phi^{-1}\) contracts \(F\).
The map \(\phi\) is a \defi{birational contraction} if it is birational and does not extract any divisors.
}
\end{definition}

\begin{definition}
{\em 
    An lc pair \((X,B)\) is \defi{divisorially log terminal (dlt)} if the coefficients of \(B\) are at most one and if there is an open subset \(U \subset X\) such that \begin{enumerate} \item \(U\) is smooth and \(B|_U\) is an snc divisor, and \item if \(E\) is a divisor over \(X\) with \(a_E(X, B) = 0\), then the center of \(E\) on \(X\) intersects \(U\). \end{enumerate}

    If \((X, B)\) is an lc pair, a \defi{\(\bb Q\)-factorial dlt modification} of \((X,B)\) is a projective birational morphism \(\pi\colon Y \to X\) from a \(\bb Q\)-factorial normal variety \(Y\) such that \({\rm Ex}(\pi)\) is a divisor, \(\pi\) only contracts log canonical places, and the log pullback \((Y, B_Y)\) is dlt. A \(\bb Q\)-factorial dlt modification exists for any lc pair by \cite[Theorem 3.1]{KollarKovacs10}.
    }
\end{definition}

\begin{definition}\label{defn:fano-type-log-CY-pair}
{\em 
    Let \(X\) be a normal projective variety.
    \begin{enumerate}
    \item We say \(X\) is \defi{of Fano type} if there exists an effective \(\bb Q\)-divisor \(B\) such that \((X, B)\) is a klt pair and \(-(K_X+B)\) is big and semiample.
    \item A \defi{log Calabi--Yau (log CY) pair} is an lc pair \((X,B)\) such that \(K_X + B \sim_{\bb Q} 0\).
    \end{enumerate}
}
\end{definition}

\(X\) is Fano type if and only if there exists an effective \(\bb Q\)-divisor \(B\) on \(X\) such that \((X,B)\) is a klt pair and \(-(K_X+B)\) is ample \cite[Lemma-Definition 2.6]{ProkhorovShokurov09}.
By \cite{FG14} an lc pair \((X,B)\) is log CY if and only if \(K_X + B \equiv 0\).

\begin{definition}
{\em
Two pairs \((X_1, B_1)\) and \((X_2, B_2)\) are \defi{crepant birational equivalent}, written \((X_1, B_1)\simeq_{\rm cbir} (X_2, B_2)\), if there exist projective birational morphisms \(f_i\colon Y \to X_i\) such that the log pullbacks of \(B_i\) on \(Y\) are equal.
}
\end{definition}

\subsection{Toric log Calabi--Yau pairs}
We will next define toric and cluster type log Calabi--Yau pairs. Before this, we first prove a proposition regarding a statement mentioned in the introduction: 
\begin{proposition}\label{prop:toric-covering}
Let $T$ be a normal quasi-projective toric variety of dimension $n$. Then, there exists a torus-invariant closed subset $Z\subset T$, of codimension at least two, such that $T\setminus Z$ is covered by copies of $\mathbb{G}_m^n$. 
\end{proposition}

\begin{proof}
Let $\Sigma$ be the fan in $\qq^n$ defining $T$, i.e., 
$T\simeq X(\Sigma)$.  By~\cite[Theorem 3.2.6(a)]{CLS11}, \(T\) has finitely many torus-invariant closed subsets of codimension at least two, so their union \(Z\) is a closed subset of \(T\) of codimension at least two.
We argue that $T\setminus Z$ is covered by copies of $\mathbb{G}_m^n$.
Let $\Sigma(1)$ be the set of rays of $\Sigma$,
and for $\rho\in \Sigma(1)$, let $D_\rho$ be the associated torus-invariant prime divisor. Now,
fix an order $\rho_1,\dots,\rho_k$ of the elements of $\Sigma(1)$. 
Then, for each $j \in \{1,\dots,k\}$, we have 
\[
T \setminus \bigcup_{i\neq j} D_{\rho_i} = \bigcup_{\tau \preceq \rho_i} O(\tau) \simeq \mathbb{G}_m^{n-1} \times \mathbb{A}^1
\]
by~\cite[Theorem 3.2.6(c)]{CLS11}
because \(\rho_i\) is a ray. This implies that $T\setminus Z$ is covered by copies of $\mathbb{G}_m^n$. 
\end{proof}

\begin{definition}\label{def:toric-pair}
{\em 
A pair $(X,B)$ is said to be \defi{toric}
if $X$ is a projective toric variety
and $B$ is a torus-invariant divisor.
}
\end{definition}

Recall that if \(X\) is a normal toric variety and \(B\) is the reduced sum  of the torus-invariant divisors, then \(K_X + B \sim 0\)~\cite[Theorem 8.2.3]{CLS11} and the pair \((X, B)\) is lc \cite[Proposition 3.7]{Kollar97}.
In the case of log CY pairs, the following definitions generalize the notion of being toric.

\begin{definition}[\cite{EFM24}]\label{defn:toric-pair}
{\em 
Let \((X,B)\) be a log CY pair.
\begin{enumerate}
    \item We say that $(X,B)$ is \defi{toric}
if $X$ is a normal toric variety and $B$ is the reduced
sum of the torus-invariant divisors.\footnote{For log CY pairs, this condition is equivalent to $(X,B)$ being a toric pair as in Definition~\ref{def:toric-pair}.}
\item We say that $(X,B)$ is \defi{of cluster type} if there exists a crepant birational 
map $\phi\colon (\pp^n,\Sigma^n)\dashrightarrow (X,B)$ such that $\mathbb{G}_m^n \cap {\rm Ex}(\phi)$
contains no divisors on $\mathbb{P}^n$. 
\item We say that $(X,B)$ is \defi{log rational} if $(X,B)\simeq_{\rm cbir} (T,B_T)$ for $T$ a normal projective toric variety
and $B_T$ the reduced sum of the torus-invariant divisors;
this condition is equivalent to $(X,B) \simeq_{\rm cbir} (\pp^n,\Sigma^n)$.
\end{enumerate}
In a similar vein, we say that an algebraic variety $X$ is of \defi{cluster type} (resp. \defi{log rational}) if it admits a boundary divisor $B$ for which $(X,B)$ is a cluster type 
(resp. log rational) log CY pair.}
\end{definition}

Every toric log CY pair $(X,B)$ is of cluster type, as it admits a torus-equivariant 
crepant birational map $(\pp^n,\Sigma^n)\dashrightarrow (X,B)$ that is an isomorphism
on $\mathbb{G}_m^n$.
Thus, the following implications hold for a log CY pair \((X, B)\):
\[\text{ $(X,B)$ is toric} \implies \text{ $(X,B)$ is of cluster type} \implies \text{$(X,B)$ is log rational}.\]

The reverse implications do not hold; see Example~\ref{ex:cluster-not-closed} and \cite[Theorem 3.3 and Proposition 5.1]{EFM24}.
If $(X,B)$ is a log CY pair of index one, then $(X,B)$ admits a toric
model if and only if the pair $(X,B)$ has birational
complexity zero \cite[Definition 2.28 and Theorem 1.6]{MauriMoraga24}.
Since crepant birational equivalences preserve the indices of log CY pairs (see, e.g.,~\cite[Lemma 3.2]{FMM22}),
any log rational log CY pair necessarily has index one:
\begin{lemma}\label{lem:index-one}
If $(X,B)$ is a log CY pair that is log rational,
then $K_X+B\sim 0$. In particular, $B$ is a reduced Weil divisor.
\end{lemma}

We now state several results about toric log CY pairs that we will use in the proofs of Theorems~\ref{thm:toric-Fano} and~\ref{thm:cluster-Fano}. First, we recall some invariants of pairs that
Brown--M\textsuperscript{c}Kernan--Svaldi--Zong \cite{BMSZ18} used to characterize toric log CY pairs:

\begin{definition}[{Cf. \cite[Definition 1.1]{BMSZ18}}]\label{defn:complexity}
{\em 
Let \((X,B)\) be a pair.
\begin{enumerate}
\item A \defi{decomposition} \(\Sigma\) of \(B\) is a finite formal sum \(\sum_{i=1}^k \alpha_i B_i = B\) where each \(\alpha_i \geq 0\) and \(B_i\) is a reduced (not necessarily irreducible) effective divisor. \item The \defi{Picard rank} of \(\Sigma\) is \(\rho(\Sigma)\coloneqq\dim(\mathrm{span}\{B_i \mid 1 \leq i \leq k\})\) where the span is inside \(\Cl_{\bb Q} X \coloneqq \Cl X\otimes_{\bb Z}\bb Q\). The \defi{norm} of \(\Sigma\) is \(|\Sigma| \coloneqq \sum_{i=0}^k \alpha_i\). The \defi{complexity} of the decomposition \((X,B;\Sigma)\) is \(c(X,B;\Sigma) \coloneqq \dim X + \rho(\Sigma) - |\Sigma|.\)\end{enumerate}
}
\end{definition}

If \((X,B)\) is a toric log CY pair and \(\Sigma\) is the decomposition of \(B\) given by the sum of its prime divisors, then \(c(X,B;\Sigma) = 0\) (see, e.g., \cite[page 2]{BMSZ18}).

\begin{theorem}[{\cite[Theorem 1.2]{BMSZ18}}]\label{thm:BMSZ-toric} Let \((X,B)\) be a log CY pair and \(\Sigma\) a decomposition of \(B\). Then \(c(X,B;\Sigma) \geq 0\).
If \(c(X,B;\Sigma) < 1\), then there exists a divisor \(D \geq \lfloor B \rfloor\) such that \((X, D)\) is a toric log CY pair, and all but \(\lfloor 2 c(X,B;\Sigma) \rfloor\) components of \(D\) are elements of the set \(\{B_i \mid 1 \leq i\leq k\}\).
In particular, \((X, \lfloor B\rfloor)\) is a (not necessarily log CY) toric pair.
\end{theorem}

The following lemma is well known to the experts (see, e.g.,~\cite[Lemma 2.3.2]{BMSZ18}).

\begin{lemma}\label{lem:toric-contr}
Let $(X, B)$ be a projective toric log CY pair, and
let $\pi\colon X\dashrightarrow Y$ be a birational contraction
to a projective variety.
Then $(Y, \pi_* B)$ is a toric log CY pair,
and $\pi$ is a toric birational map.
\end{lemma}

The following result, which shows that certain birational modifications of toric log CY pairs are toric, will be useful in the proofs of both Theorems~\ref{thm:toric-Fano} and~\ref{thm:cluster-Fano}.
\begin{lemma}\label{lem:toric-mod}
Let $(X,B)$ be a projective toric log CY pair,
let $p\colon Y\dashrightarrow X$ be a birational contraction that only contracts non-canonical places of $(X,B)$,
and let $(Y,B_Y)$ be the log pullback of $(X,B)$ to $Y$.
Then $(Y,B_Y)$ is a toric log CY pair, and $Y\dashrightarrow X$ is a toric projective birational contraction.
\end{lemma}

\begin{proof}
The pair $(X,B)$ has index one by~\cite[Theorem 8.2.3]{CLS11},
so the assumption on $p$ implies it only contracts log canonical places of $(X,B)$.
In particular, we may find a $\qq$-factorial dlt modification $(Z,B_Z)\rightarrow (X,B)$
such that $Z\dashrightarrow Y$ is a birational contraction.
Let $\Sigma$ (resp. \(\Sigma_Z\)) be the decomposition of $B$ (resp. \(B_Z\)) given by the sum of its prime divisors.
Then $c(X,B;\Sigma)=0$ since \((X,B)\) is toric. Since
\(\rho(\Sigma_Z)=\rho(\Sigma)+\rho(Z/X)\)
and \(|\Sigma_Z|=|\Sigma|+\rho(Z/X)\),
we get 
\(c(Z,B_Z;\Sigma_Z)=0\),
so $(Z,B_Z) = (Z, \lfloor B_Z \rfloor)$ is a toric log CY pair by Theorem~\ref{thm:BMSZ-toric}.
Applying Lemma~\ref{lem:toric-contr} to \(Z \dashrightarrow Y \dashrightarrow X\), we conclude that $(Y,B_Y)$ is a toric log CY pair, and hence that $Y\dashrightarrow X$ is a toric birational contraction.
\end{proof}

\subsection{Families of varieties and pairs}

Next, we will consider relative versions of Definition~\ref{defn:fano-type-log-CY-pair}. A \defi{projective contraction} is a projective morphism \(f\colon X \to T\) of normal quasi-projective varieties such that \(f_* \cal O_X = \cal O_T\). A \defi{fibration} is a projective contraction whose general fiber is positive dimensional.

\begin{definition}
    Let \((X,B = \sum \alpha_i B_i)\) be a pair and \(f\colon X \to T\) a projective contraction. For a closed point \(t\in T\), the \defi{log fiber} is \((X_t, B_t \coloneqq \sum_{B_i \not\supset X_t} \alpha_i B_i|_{X_t})\).
    If \(X_t\) is normal and is not contained in any prime divisor in the support of \(B\), then \((X_t, B_t)\) is a pair.
\end{definition}
A \defi{family of varieties} is a flat projective contraction \(f\colon X \to T\) of normal quasi-projective varieties.
A \defi{family of pairs} is a family of varieties \(X \to T\) and a pair \((X, B)\) such that the log fiber over any closed point is a pair.

\begin{definition}\label{defn:Fano-type-log-CY-morphism}
{\em \hfill
\begin{enumerate}
    \item A \defi{family of Fano varieties} is a family of varieties such that the fiber over every closed point is a Fano variety.
    \item A \defi{Fano fibration} is a projective contraction \(X \to T\) such that \(-K_X\) is ample over \(T\).
    \item A \defi{Fano type morphism} is a projective contraction \(X \to T\) such that there exists a
    \(\bb Q\)-divisor \(B\) on \(X\) with \((X,B)\) a klt pair, \(B\) big over \(Z\), and \(K_X + B \sim_{T, \qq} 0\).
    \item A \defi{family of log CY pairs} is a family of pairs such that the log fiber over any closed point is a log CY pair.
    \item A \defi{log CY fibration} is a projective contraction $X\rightarrow T$ such that there exists a $\qq$-divisor $B$ on $X$ with $(X,B)$ an lc pair and $K_X+B\sim_{T,\qq} 0$.
\end{enumerate}
}
\end{definition}

Note that a family of Fano varieties is not necessarily a Fano fibration, but it becomes a Fano fibration after a finite \'etale base change, since the geometric generic fiber is Fano. On the other hand, a Fano fibration is not necessarily a family of Fano varieties over \(T\), but it is over some dense open subset of the base. The same properties hold for log CY pairs.

\begin{lemma}\label{lem:FT1}
Let $X\rightarrow T$ be a Fano type morphism.
Let \(Y\) be a normal quasi-projective variety and $Y\rightarrow X$ a projective birational morphism
that only contracts non-terminal places of $X$.
Then, the composition $Y\rightarrow T$ is a Fano type morphism.
\end{lemma}

\begin{proof}
Let $\Delta$ be a boundary on $X$ such that
$(X,\Delta)$ is klt and $-(K_X+\Delta)$ is big and semiample over $T$.
The log pullback $\Delta_Y$ of $\Delta$ is effective
as $Y\rightarrow X$ only
contracts non-terminal places of $(X,\Delta)$,
so \((Y,\Delta_Y)\) is a klt pair.
Since $-(K_Y+\Delta_Y)$ is big and semiample over $T$,
the morphism $Y\rightarrow T$ is of Fano type.
\end{proof}

The following lemma is well known (see, e.g.,~\cite[Lemma 2.17(3)]{Mor24a}).

\begin{lemma}\label{lem:FT2}
Let $X\rightarrow T$ be a Fano type morphism, and
let $(X,B)$ be a pair that is log CY over $T$, i.e., \(K_X + B \sim_{T, \bb Q} 0\).
Let \(Y\) be a normal quasi-projective variety and $Y\rightarrow X$ a projective birational morphism
that only contracts non-canonical places of $(X,B)$.
Then, the composition $Y\rightarrow T$ is a morphism of Fano type.
\end{lemma}

For a rational map \(\pi\colon X \dashrightarrow Z\) of normal projective varieties, we define the pullback of divisors as follows. Let \(\widetilde{X}\) be a normal variety that resolves \(\pi\), with projections \(p_1\colon\widetilde{X} \to X\) and \(p_2\colon\widetilde{X} \to Z\). Let \(D_Z\) be a \(\bb Q\)-Cartier divisor on \(Z\), and assume \(m D_Z\) is Cartier. Then we define the \(\bb Q\)-divisor \[\pi^* D_Z \coloneqq \frac{1}{m} {p_1}_* p_2^* (mD_Z),\]
where \(p_2^* (mD_Z)\) is the Weil divisor associated to the Cartier divisor corresponding to \(m D_Z\) under the isomorphism 
\[
H^0(Z, \cal O_Z(mD_Z)) \xrightarrow{\cong} H^0(Z, {p_2}_* p_2^*\cal O_Z(mD_Z)) = H^0(\widetilde{X}, p_2^*\cal O_Z(mD_Z)).
\]

If \(D_X\) is a prime divisor on \(X\), we define its pushforward \(\pi_* D_X\) to be its strict transform on \(Z\). We extend \(\pi_*\) linearly to Weil divisors on \(X\). If \(\pi\) is a birational contraction, then the equality \(\pi_* \pi^* D_Z = D_Z\) holds as \(\bb Q\)-divisors on \(Z\).
Furthermore, if \(\pi\) is a birational contraction, then \(\pi_*\) induces a homomorphism \(\Cl X \to \Cl Z\).

We now introduce some notions of class groups in families.

\begin{definition}\label{defn:class-group-restrictions}
{\em 
Let $f\colon X\rightarrow T$ be a fibration of \(\bb Q\)-factorial normal varieties.
We say that $f$ has \defi{surjective class group restrictions} (resp. \defi{injective class group restrictions}, \defi{isomorphic class group restrictions})
if the restriction homomorphism
\[
{\rm Cl}_{\bb Q}(X/T) \coloneqq (\Cl X / f^* \Pic T) \otimes_{\bb Z} \bb Q \rightarrow {\rm Cl}_{\bb Q}(X_t) \coloneqq \Cl (X_t) \otimes_{\bb Z} \bb Q
\]
is surjective (resp. injective, an isomorphism) for every closed point $t\in T$.
}
\end{definition}

The property of injective class group restrictions descends under birational contractions:

\begin{lemma}\label{lem:injective-class-groups}
Let \(\phi\colon X \to T\) and \(\psi\colon Z \to T\) be fibrations of \(\bb Q\)-factorial varieties
such that the fibers $X_t$ and $Z_t$ are normal for every \(t \in T\).
Assume there is a birational contraction
\(\pi\colon X\dashrightarrow Z\) over \(T\).
If \(\phi\) has injective class group restrictions,
then $\psi$ also
has injective class group restrictions.
\end{lemma}

\begin{proof}
Let $D$ be a \(\bb Q\)-divisor on $Z$
for which $D|_{Z_t}\sim_\qq 0$ for some $t\in T$.
The pullback $\pi^*D$ to $X$ satisfies
$\pi^*D|_{X_t}\sim_\qq 0$,
so $\pi^*D\sim_{\qq,T} 0$
because $X\rightarrow T$ has injective class group restrictions.
Since \(\pi\) is a birational contraction, we conclude
$D = \pi_*\pi^*D \sim_{\qq,T} 0$.
\end{proof}

The following properties of class groups of fibers of Fano type morphisms will be useful.

\begin{lemma}\label{lem:fano-type-num-triv-nbhd}
    Let \(f\colon X \to T\) be a Fano type morphism of \(\bb Q\)-factorial varieties.
    \begin{enumerate}
        \item\label{item:fano-type-num-triv-nbhd-part-1} Let \(t_0\in T\) be a closed point in the smooth locus of \(T\). There is a smooth open neighborhood \(t_0\in U\subset T\) such that for any divisor \(D\) on \(X\) over \(T\), if \(D|_{X_{t_0}} \equiv 0\), then \(D\) is \(\bb Q\)-linearly trivial over an open neighborhood of \(t_0\).\footnote{That is, there exists an open neighborhood \(U \ni t_0\) such that \(D|_{X_U} \sim_{\bb Q} 0\) in \(\Cl(X_U/U) \coloneqq \Cl X_U / f|_{X_U}^*\Pic U \).}
        \item\label{item:fano-type-num-triv-nbhd-part-2} There exists a nonempty open subset \(V \subset T\) such that
        \(f|_{X_V}\colon X_V \to V\) has injective class group restrictions.
        \item\label{item:fano-type-num-triv-nbhd-part-3} If $f$ has surjective class group restrictions, then \(f\) has isomorphic class group restrictions
        over a nonempty open subset of $T$.
    \end{enumerate}
\end{lemma}

\begin{proof}
    First, we note that for any divisor \(D\) on \(X\) over \(T\) that is numerically trivial on \(X_{t_0}\), there is an open neighborhood \(U_D\ni t_0\) over which \(D|_{X_{U_D}} \equiv 0\). Indeed, over any affine open subset of \(T\), we may run a \(D\)-MMP for \(X\) that terminates with a good minimal model by \cite[Corollary 1.3.2]{BCHM}, and the assumption that \(D\) is numerically trivial on \(X_{t_0}\) implies that this MMP is trivial on an open neighborhood of \(t_0\). Since \(f\) is of Fano type and \(D\) is numerically trivial over this open neighborhood, it is \(\bb Q\)-linearly trivial over this open neighborhood.
    
    For~\eqref{item:fano-type-num-triv-nbhd-part-1}, since \(X\to T\) is Fano type, \(\Cl_{\bb Q}(X/T) \cong \Cl_{\bb Q}(X/T)/\equiv_{\bb Q, T}\) is finitely generated (see \cite[Theorem 3.18 and Definition 3.13]{BraunMoraga24}).
    Thus, after replacing \(T\) by an open neighborhood of \(t_0\) finitely many times, the conclusion holds.

    For~\eqref{item:fano-type-num-triv-nbhd-part-2}, let \(t_0\in T\), and let \(t_0 \in U_0\) be the open neighborhood from~\eqref{item:fano-type-num-triv-nbhd-part-1}. Then \(\ker(\Cl_{\bb Q}(X/T) \to \Cl_{\bb Q}(X_{t})) \subseteq \ker(\Cl_{\bb Q}(X/T) \to \Cl_{\bb Q}(X_{t_0}))\) for any \(t\in U_0\). If the containment is strict for some \(t_1\in U_0\), apply~\eqref{item:fano-type-num-triv-nbhd-part-1} to find an open neighborhood \(t_1 \in U_1 \subset U_0\). Repeating this process, we obtain a sequence \(\ker(\Cl_{\bb Q}(X/T) \to \Cl_{\bb Q}(X_{t_{i+1}})) \subseteq \ker(\Cl_{\bb Q}(X/T) \to \Cl_{\bb Q}(X_{t_i}))\) of subspaces of the finite-dimensional \(\bb Q\)-vector space \(\Cl_{\bb Q}(X/T)\). This must stabilize after a finite number of steps, so after shrinking \(T\) a finite number of times, we obtain the open subset \(V\) in~\eqref{item:fano-type-num-triv-nbhd-part-2}.
    Finally, \eqref{item:fano-type-num-triv-nbhd-part-3} is immediate from~\eqref{item:fano-type-num-triv-nbhd-part-2}.
\end{proof}

\subsection{Properties of cluster type log Calabi--Yau pairs and dlt modifications}

Now we focus on log CY pairs of cluster type and prove several results that we will use in the proof of Theorem~\ref{thm:cluster-Fano}.
The next lemma is one of the reasons for the terminality assumption the theorem.

\begin{lemma}\label{lem:div-combination}
Let $(X,B)$ be a pair of index one,
and let $(X,\Delta)$ be a terminal pair.
For every $\epsilon \in (0,1)$,
the exceptional non-terminal places of $(X,(1-\epsilon)B+\epsilon \Delta)$
are log canonical places of $(X,B)$.
\end{lemma}

\begin{proof}
If $E$ is an exceptional non-terminal place
of $(X,(1-\epsilon)B+\epsilon \Delta)$, then
\[
a_E(X,(1-\epsilon)B+\epsilon \Delta) = 
(1-\epsilon)a_E(X,B) + \epsilon a_E(X,\Delta) \leq 1.
\]
As $a_E(X,\Delta)>1$, we have $\epsilon a_E(X,\Delta)>\epsilon$
and so $(1-\epsilon)a_E(X,B)< 1-\epsilon$.
Thus, $a_E(X,B)<1$.
Since $(X,B)$ has index one, we get $a_E(X,B)=0$,
so $E$ is a log canonical place of $(X,B)$.
\end{proof}

Next, we show that a dlt modification of a cluster type log CY pair is also of cluster type.

\begin{lemma}\label{lem:cluster-type-extraction}
Let $(X,B)$ be a log CY pair of cluster type.
Let $\phi\colon Y \dashrightarrow X$ be a birational map
that only extracts log canonical places of $(X,B)$,
and only contracts non-terminal places of $(X,B)$.
Let $(Y,B_Y)$ be the log pullback of $(X,B)$ to $Y$.
Then $(Y,B_Y)$
is a log CY pair of cluster type.
\end{lemma}

\begin{proof}
First, the assumption that \(\phi\) only contracts non-terminal places implies that \(B_Y\) is effective, so \((Y, B_Y)\) is a pair.
Let $\pi\colon (\pp^n,\Sigma^n)\dashrightarrow (X,B)$ be a crepant birational map
such that $\mathbb{G}_m^n \cap {\rm Ex}(\pi)$ has codimension at least two on the torus.
By the assumption that \(\phi\) only extracts log canonical places, the divisorial part of the  exceptional locus of $\phi^{-1}\colon X\dashrightarrow Y$
is contained in \({\rm supp}(B)\).
We conclude that 
$\phi^{-1} \circ \pi \colon (\pp^n,\Sigma^n)\dashrightarrow (Y,B_Y)$ is a crepant birational map for which 
$\mathbb{G}_m^n \cap {\rm Ex}(\phi^{-1}\circ \pi)$ has codimension at least two in the torus.
So $(Y,B_Y)$ is of cluster type.
\end{proof}

The property of being a cluster type log CY pair also descends under dlt modifications.

\begin{lemma}\label{lem:cluster-type-descend-dlt}
Let $(Y,B_Y)$ be a log CY pair
of cluster type.
If $\phi\colon Y\dashrightarrow X$ is a birational contraction that only contracts 
log canonical places of $(Y,B_Y)$,
then $(X,\phi_* B_Y)$ is a log CY pair of cluster type.
\end{lemma}

\begin{proof}
Write \(B \coloneqq \phi_* B_Y\).
First note that \(\phi\colon (Y,B_Y) \dashrightarrow (X, B)\) is crepant birational by the negativity lemma.
Let $\psi\colon (\pp^n,\Sigma^n)\dashrightarrow (Y,B_Y)$ be a crepant birational map
such that $\mathbb{G}_m^n\cap {\rm Ex}(\psi)$
has codimension at least two in $\mathbb{G}_m^n$.
Then, as $\phi$ only contracts log canonical places of $(Y,B_Y)$, the composition 
$\phi\circ \psi\colon (\pp^n,\Sigma^n)\dashrightarrow (X,B)$ is a crepant birational map
such that $\mathbb{G}_m^n \cap {\rm Ex}(\phi\circ \psi)$ has codimension at least two in $\mathbb{G}_m^n$.
\end{proof}

Next, we construct dlt modifications with certain desirable properties. The nice dlt modifications constructed in the following lemmas will be useful in giving another characterization of cluster type log CY pairs (Lemma~\ref{lem:dlt-mod-cluster}) and in proving Theorem~\ref{thm:cluster-Fano}. First, we need some definitions.

\begin{definition}\label{def:toroidal-blow-ups}
{\em 
Let $(X,B)$ be a dlt pair 
and $X^{\rm snc}$ be the largest open subset of $X$
on which the pair is simple normal crossing.
Let $B^{\rm snc}$ be the restriction of $B$ to $X^{\rm snc}$.
\begin{enumerate}
\item
A \defi{formally toric blow-up}
of $(X^{\rm snc},B^{\rm snc})$
is a blow-up $Y^{\rm snc} \rightarrow X^{\rm snc}$ that is a formally toric morphism over any closed point of $(X^{\rm snc},B^{\rm snc})$ (see \cite[Definition 3.42]{MS21} for more details).
We say that a formally toric blow-up is a \defi{formally toric dlt blow-up}
if the log pullback of $(X^{\rm snc},B^{\rm snc})$ to $Y$ is a dlt pair.
\item
A \defi{formally toric dlt blow-up} of $(X,B)$ is a dlt modification
$(Y,B_Y)\rightarrow (X,B)$ that restricts to a formally toric blow-up 
of $(X^{\rm snc},B^{\rm snc})$.
We say that a formally toric blow-up
is a \defi{blow-up of a stratum} if it is induced by the blow-up of the 
reduced scheme structure of a stratum of $(X^{\rm snc},B^{\rm snc})$.
\end{enumerate}
}    
\end{definition}

\begin{remark}\label{rem:toroidal-blow-up}
{\em 
\cite[Proof of Proposition 37]{dFKX17} shows
that if $(X,B)$ is a dlt pair, then every formally toric dlt blow-up of
$(X^{\rm snc},B^{\rm snc})$ extends to a formally toric blow-up
of $(X,B)$.
Thus, every stratum of $(X,B)$ induces a formally toric blow-up of the stratum, and this formally toric dlt blow-up of \((X,B)\) is unique up to a small birational transformation over the base.
}    
\end{remark}

\begin{lemma}\label{lem:nice-dlt-mod-wrt-E}
Let $(X,B)$ be a log canonical pair,
and let $E\subset X$ be an effective \(\bb Q\)-Cartier \(\bb Q\)-divisor that 
has no common components with $B$.
Then there is a \(\bb Q\)-factorial dlt modification $\phi\colon (Y,B_Y)\rightarrow (X,B)$ such that the strict transform $\phi^{-1}_*E$ of \(E\) on \(Y\)
contains no log canonical centers of $(Y,B_Y)$. Furthermore, the pair \((Y,B_Y+ \epsilon \phi_*^{-1} E)\) is dlt for sufficiently small \(\epsilon > 0\).

If moreover $(X,B)$ is a dlt pair, then $\phi$ may be obtained by a sequence of formally toric blow-ups of strata of $(X,B)$. 
\end{lemma}

\begin{proof}
Consider the pair $(X,B+\epsilon E)$, which is not necessarily log canonical.
By linearity of log discrepancies with respect to the boundary divisor,
for $\epsilon>0$ small enough, every non-lc center of $(X,B+\epsilon E)$
is a log canonical center of $(X,B)$.
By~\cite[Theorem 3.1]{KollarKovacs10}
there exists a \(\bb Q\)-factorial dlt modification
$\phi \colon (Y,B_Y+\epsilon \phi^{-1}_* E)\rightarrow (X,B+\epsilon E)$ where
$B_Y$ is the sum of
\(\phi^{-1}_* B\) and the reduced exceptional divisor of $\phi$ (see \cite[Definitions and Notation 1.9]{KollarKovacs10}).
Then \((Y, B_Y)\) is dlt, and
by the previous considerations, we have 
$\phi^*(K_X+B)=K_Y+B_Y$.
Since $(Y,B_Y+\epsilon \phi^{-1}_* E)$
is log canonical, we know that  $\phi^{-1}_* E$ contains no log canonical centers of $(Y,B_Y)$.

Now, assume in addition that $(X,B)$ is dlt.
By the previous paragraph,
there exists a dlt modification $\phi_0\colon (Y_0,B_{Y_0})\rightarrow (X,B)$
such that ${\phi_0^{-1}}_*E$
contains no log canonical centers of $(Y_0,B_{Y_0})$.
We first argue that the dlt modification $(Y_0,B_{Y_0})$ induces a formally toric dlt blow-up
$(Y_0^{\rm snc},B_{Y_0^{\rm snc}})\rightarrow (X^{\rm snc},B^{\rm snc})$.
Locally at every closed point, the pair $(X^{\rm snc},B^{\rm snc})$ has local complexity zero \cite[Example 3.45, Lemma 3.41, and Definition 3.28]{MS21}.
The morphism $(Y_0^{\rm snc},B_{Y^{\rm snc}_0})\rightarrow (X^{\rm snc},B^{\rm snc})$ only contracts
divisors with coefficient one, so the complexity of this morphism is zero at every closed point of $X^{\rm snc}$ (see~\cite[Definition 3.15]{MS21}).
Hence,~\cite[Theorem 1]{MS21} implies that the dlt modification
$(Y_0^{\rm snc},B_{Y^{\rm snc}_0})\rightarrow (X^{\rm snc},B^{\rm snc})$ is formally toric over the base.
That is, \(\phi_0\) is a formally toric dlt blow-up
along some ideal sheaf whose support is contained in \({\rm supp}(B^{\rm snc})\).
Now we will construct \(Y\) as a higher model of \(Y_0^{\rm snc}\) that is obtained from \((X,B)\) by a sequence of formally toric dlt blow-ups along strata.

We first construct this model over formal neighborhoods in \(X^{\rm snc}\).
Here the morphism $Y_0^{\rm snc}\rightarrow X^{\rm snc}$
corresponds to unimodular fan refinements of unimodular cones
of $(X^{\rm snc},B^{\rm snc})$ \cite[Theorem 1.3.12 and Exercise 9.3.7]{CLS11}.
Every refinement of a unimodular cone $\sigma$ is refined by a sequence of stellar subdivisions of $\sigma$ \cite[Proposition 11.1.7]{CLS11}.
Hence, over formal neighborhoods in \(X^{\rm snc}\), there is a formally toric dlt blow-up 
$(Y^{\rm snc},B_Y^{\rm snc})\rightarrow (X^{\rm snc},B^{\rm snc})$
that is obtained by successively blowing up strata of $(X^{\rm snc}, B^{\rm snc})$
\cite[Proposition 3.3.15]{CLS11}
and such that there is a projective birational
morphism $Y^{\rm snc}\rightarrow Y_0^{\rm snc}$.
Furthermore, additional blow-ups along strata of \((X^{\rm snc},B^{\rm snc})\) result in a higher model that is still dlt and formally toric over \((X^{\rm snc},B^{\rm snc})\).
Now, consider the locally closed decomposition \(X^{\rm snc} = \bigsqcup_{i \in I} Z_i\) defined by the strata of \(B\). For each \(i\in I\), the above procedure produces the same sequence of blow-ups over any closed point of the locally closed subset \(Z_i\). Since the index set \(I\) is finite,
we may apply the above procedure to each \(Z_i\) (performing additional blow-ups along strata if necessary) to get a sequence of blow-ups of strata of \((X^{\rm snc},B^{\rm snc})\), which yields a formally toric dlt blow-up $(Y^{\rm snc},B_Y^{\rm snc})\rightarrow (X^{\rm snc},B^{\rm snc})$ that factors through a projective birational
morphism $Y^{\rm snc}\rightarrow Y_0^{\rm snc}$.
By Remark~\ref{rem:toroidal-blow-up}, this extends to a formally toric dlt blow-up
$(Y,B_Y)\rightarrow (X,B)$.
By construction, the induced birational map
$\psi \colon Y\dashrightarrow Y_0$ is a birational contraction
and ${\rm Ex}(\psi)$ contains no log canonical centers of $(Y,B_Y)$.
Then, the strict transform of $E$ on $Y$ contains no log canonical centers of $(Y,B_Y)$.
\end{proof}

We can now give the following equivalent characterization of cluster type log CY pairs.

\begin{lemma}\label{lem:dlt-mod-cluster}
A log CY pair $(X,B)$
is of cluster type if and only if 
there exists a \(\bb Q\)-factorial dlt modification $(Y,B_Y)\rightarrow (X,B)$
with a crepant birational contraction
$\pi\colon (Y,B_Y) \dashrightarrow (\pp^n,\Sigma^n)$.
Furthermore, if $E \coloneqq {\rm Ex}(\pi)\setminus {\rm supp}(B_Y)$, we may assume that
$(Y,B_Y+\epsilon E)$ is a dlt pair for $\epsilon > 0$ small enough.
\end{lemma}

\begin{proof}
First, assume that $(X,B)$ is of cluster type.
Then there exists a crepant birational map
$(\pp^n,\Sigma^n)\dashrightarrow (X,B)$
that only contracts log canonical places.
In particular, there exists a \(\bb Q\)-factorial dlt modification $(Z,B_Z)\rightarrow (X,B)$
such that the composition $p\colon Z\dashrightarrow \pp^n$ is a birational contraction. Let $E_Z \coloneqq {\rm Ex}(p)\setminus {\rm supp}(B_Z)$.
Applying Lemma~\ref{lem:nice-dlt-mod-wrt-E} to \((Z, B_Z)\) and $E_Z$
yields a dlt modification $q\colon (Y,B_Y)\rightarrow (Z,B_Z)$ such that $(Y,B_Y+ \epsilon q^{-1}_*E_Z)$ is dlt for $\epsilon>0$ small enough.
Denote the composition by \(\pi \coloneqq p \circ q\colon Y \dashrightarrow \bb P^n\).
Then $q^{-1}_*E_Z={\rm Ex}(\pi)\setminus {\rm supp}(B_Y)$,
so \((Y, B_Y) \to (X,B)\) is the desired dlt modification.

Now assume $(X,B)$ admits a \(\bb Q\)-factorial dlt modification $\phi\colon (Y,B_Y)\rightarrow (X,B)$
with a crepant birational contraction $\pi\colon (Y,B_Y)\dashrightarrow (\pp^n,\Sigma^n)$.
Then $(X,B)$ is of cluster type by Lemma~\ref{lem:cluster-type-descend-dlt}.
Furthermore, if $E = {\rm Ex}(\pi)\setminus {\rm supp}(B_Y)$, then by Lemma~\ref{lem:nice-dlt-mod-wrt-E} we may replace $(Y,B_Y)$ with a higher $\qq$-factorial dlt modification
such that $(Y,B_Y+\epsilon E)$ is dlt for $\epsilon>0$ small enough.
\end{proof}

The next lemma will be used to ensure that certain linear combinations of pairs are terminal on higher dlt modifications.

\begin{lemma}\label{lem:comb-terminal-dlt}
Let $(X,B)$ be a dlt pair of index one.
Let $(X,\Delta)$ be a terminal pair, and assume that the prime components of \(\Delta\) are \(\bb Q\)-Cartier.
Let $(Z,B_Z)\rightarrow (X,B)$ be any dlt modification.
Then for any sufficiently small $\epsilon > 0$, there exists a \(\bb Q\)-factorial dlt modification $(Y,B_Y)\rightarrow (Z,B_Z)$
such that the log pullback of
$(X,(1-\epsilon)B+\epsilon \Delta)$ to $Y$
is a terminal pair.
\end{lemma}

\begin{proof}
Let $D$ be the (reduced) sum of the prime components of $\Delta$
that are not components of $B$.
By Lemma~\ref{lem:nice-dlt-mod-wrt-E}, there is a \(\bb Q\)-factorial dlt modification $\phi\colon (W,B_W)\rightarrow (Z,B_Z)$
such that
\(\phi^{-1}_* D\)
contains no log canonical centers of $(W,B_W)$.
For sufficiently small \(\epsilon > 0\),
every divisor contracted by $W\rightarrow X$
is a non-terminal place of $(X,(1-\epsilon)B+\epsilon\Delta)$,
so the log pullback $(W,\Gamma_W)$ of $(X,(1-\epsilon)B+\epsilon\Delta)$ to $W$ is a pair; furthermore, the pair $(W,\Gamma_W)$ is klt.

By Lemma~\ref{lem:div-combination},
every exceptional non-terminal place of $(X,(1-\epsilon)B+\epsilon\Delta))$ is a log canonical place of $(X,B)$.
In particular, every exceptional non-terminal place of $(W,\Gamma_W)$ is a log canonical place of $(W,B_W)$,
and the log canonical centers of \((W,B_W)\) are precisely its proper strata.
By construction, the components of $\Gamma_W$ that are not contained in $B_W$ are precisely the components of \(\phi^{-1}_* D\).
So, on a neighborhood of
the generic point of every proper stratum of $(W,B_W)$,
the support of $\Gamma_W$ is contained
in the support of $B_W$.
Since $(W,B_W)$ is log smooth along the generic point of each stratum,
we conclude that $(W,\Gamma_W)$ is log smooth around every non-terminal center.
Therefore, we may obtain a terminal model of $(W,\Gamma_W)$
by successively blowing up strata of $(W,B_W)$ (see \cite[Lemma 2.29 and Proposition 2.36(2)]{KM98}).
Hence, the composition $(Y,B_Y)\rightarrow (W,B_W)\rightarrow (Z,B_Z)$ is a \(\bb Q\)-factorial dlt modification 
for which the log pullback of
\((X,(1-\epsilon)B+\epsilon\Delta)\) is terminal.
\end{proof}

Finally, we will apply the following two technical lemmas to certain fibrations in Section~\ref{sec:cluster-type} to show that, after shrinking the base, the outcomes of certain MMPs will have nice fibers.

\begin{lemma}\label{lem:nice-fibers-klt-CY-fibration}
Let $(Y,\Gamma)$ be a terminal pair and $f\colon (Y,\Gamma)\rightarrow T$ a log Calabi-Yau fibration. 
Let \(U \coloneqq U_1 \cap U_2 \cap U_3\) be the intersection of the following dense open subsets of \(T\):
\begin{enumerate}[label=(\roman*)]
    \item the smooth locus $U_1 \subset T$, \label{item:nice-fibers-klt-CY-fibration-U_1}
    \item an open set $U_2$ over which all log fibers $(Y_t,\Gamma_t)$ are terminal and $\Gamma_t = \Gamma|_{Y_t}$, and\label{item:nice-fibers-klt-CY-fibration-U_2}
    \item the complement $U_3$ of the boundary divisor and the moduli divisor induced by the canonical bundle formula
    on $T$ by $(Y,\Gamma)$. \label{item:nice-fibers-klt-CY-fibration-U_3}
\end{enumerate}
Let \(E\) be an effective $\qq$-Cartier \(\bb Q\)-divisor on $Y_U$
that contains no fiber of $f$.
Then for \(\delta > 0\) small enough and
for any sequence $Y_U\dashrightarrow W$ of steps of the $(K_{Y_U}+\Gamma_U+\delta E)$-MMP over $U$,
the following conditions hold:
\begin{enumerate}
    \item the fibers of $Y_U\rightarrow U$ are normal, \label{item:nice-fibers-klt-CY-fibration-normal-fibers-Y}
    \item the fibers of $W\rightarrow U$ are normal, \label{item:nice-fibers-klt-CY-fibration-normal-fibers-W}
    \item for every $t\in U$, the induced birational map $Y_t\dashrightarrow W_t$ is a contraction, and \label{item:nice-fibers-klt-CY-fibration-fibers-birational-contraction}
    \item if $\Gamma_W$ is the pushforward of $\Gamma$ to $W$, then the log fibers of $(W,\Gamma_W)\rightarrow U$ are terminal. \label{item:nice-fibers-klt-CY-fibration-log-fibers-klt}
\end{enumerate}
\end{lemma}

\begin{proof}
Let $d$ be the dimension of $U$, and let
$t\in U$ be a closed point.
Then~\eqref{item:nice-fibers-klt-CY-fibration-normal-fibers-Y} holds by~\ref{item:nice-fibers-klt-CY-fibration-U_2}; furthermore, if \(H_1,\dots,H_d\) are general hyperplanes on \(U\) containing \(t\), the pair
$(U,H_1+\dots+H_d;t)$ is log smooth and has $t$ as a log canonical center.
Consider the pair
\begin{equation}\label{eq:pair-pull}
(Y,\Gamma+f^*H_1+\dots+f^*H_d). 
\end{equation}
This pair is log canonical over a neighborhood of \(t\) by~\ref{item:nice-fibers-klt-CY-fibration-U_3} and the canonical bundle formula (see, e.g.,~\cite[Proposition 4.16]{Fil18}).
The pair $(Y_t,\Gamma_t)$ is terminal by~\ref{item:nice-fibers-klt-CY-fibration-U_2},
so $Y_t$ is a minimal log canonical center of~\eqref{eq:pair-pull}. Further $(Y_t,\Gamma_t)$ is the pair induced by adjunction to the minimal log canonical center:
indeed, if $(Y_t,\Gamma_t)$ is the pair obtained by adjunction of~\eqref{eq:pair-pull} to $Y_t$, then $(Y_t,\Gamma_t)$ is log CY and $\Gamma_t\geq \Gamma_t$, so $\Gamma_t=\Gamma_t$.

In what follows, we will show that there is an open neighborhood of \(t\) such that, for sufficiently small \(\delta > 0\) depending on \(t\), the conclusion holds over this neighborhood of \(t\). Since \(U\) is compact, we can find \(\delta\) that works for all \(t\in U\).

Let \(E_t\) denote the restriction of \(E\) to \(Y_t\).
By assumption, \(E\) contains no log canonical centers of~\eqref{eq:pair-pull} over a neighborhood of \(t\).
Therefore, for $\delta>0$ small enough, the pair
\begin{equation}\label{eq:pair-pull-2}
(Y,\Gamma+f^*H_1+\dots+f^*H_d+\delta E)    
\end{equation}
is also log canonical over a neighborhood of \(t\), 
and the pair $(Y_t,\Gamma_t+\delta E_t)$, which is obtained by adjunction
to the minimal log canonical center,
is terminal.
Let $\phi\colon Y_U \dashrightarrow W$ be a sequence
of steps of the $(K_{Y_U}+\Gamma_U+\delta E)$-MMP over $U$.
Then, we have a commutative diagram
\[
\xymatrix
{
(Y_U,\Gamma_U+f^*H_1+\dots+f^*H_d+\delta E)\ar@{-->}[r]^-{\phi} \ar[d]_-{f} & 
(W,\Gamma_W+{f'}^*H_1+\dots+{f'}^*H_d+\delta E_W) \ar[ld]^-{f'} \\ 
T
}
\]
where $\Gamma_W$ (resp. $E_W$) is the pushforward
of $\Gamma_U$ (resp. $E$) to $W$.
The pair 
\begin{equation}\label{eq:pair-pull-3}
(W,\Gamma_W+{f'}^*H_1+\dots+{f'}^*H_d+\delta E_W)
\end{equation}
is log canonical over a neighborhood of \(t\),
and $W_t$ is a minimal log canonical center.
By \cite{Kawamata98}, this implies that~\eqref{item:nice-fibers-klt-CY-fibration-normal-fibers-W} holds.

It remains to show~\eqref{item:nice-fibers-klt-CY-fibration-fibers-birational-contraction} and~\eqref{item:nice-fibers-klt-CY-fibration-log-fibers-klt}. First, by adjunction, the induced birational map
\begin{equation}\label{eq:neg-map}
    \phi_t\colon (Y_t, \Gamma_t+\delta E_t) \dashrightarrow (W_t,\Gamma_{W_t}+\delta E_{W_t})
\end{equation}
between minimal log canonical centers of~\eqref{eq:pair-pull-2} and~\eqref{eq:pair-pull-3}
is $(K_{Y_t}+\Gamma_t+\delta E_t)$-negative,
so the negativity lemma applied to a resolution of \(\phi_t\) shows that
\begin{equation}\label{ineq:log-disc}
a_{F_t}(W_t,\Gamma_{W_t}+\delta E_{W_t}) \geq
a_{F_t}(Y_t,\Gamma_t+\delta E_t)
\end{equation}
for any prime divisor \(F_t\).
We now show~\eqref{item:nice-fibers-klt-CY-fibration-fibers-birational-contraction}.
By contradiction, suppose \(\phi_t\) is not a birational contraction, so it extracts some prime divisor \(F_t\). As \(F_t\) is exceptional over \(Y_t\), the inequality~\eqref{ineq:log-disc} implies that
$F_t$ appears in $\Gamma_{W_t}+\delta E_{W_t}$
with negative coefficient,
contradicting the fact that $(W_t,\Gamma_{W_t}+\delta E_{W_t})$
is a pair.
This proves~\eqref{item:nice-fibers-klt-CY-fibration-fibers-birational-contraction}.
Finally, the inequality~\eqref{ineq:log-disc} and the terminality of \((Y_t,\Gamma_t+\delta E_t)\) imply that $(W_t,\Gamma_{W_t}+\delta E_t)$ is terminal and hence that $(W_t,\Gamma_{W_t})$ is terminal, so~\eqref{item:nice-fibers-klt-CY-fibration-log-fibers-klt} holds.
\end{proof}

\begin{lemma}\label{lem:nice-fibers-lc-CY-fibration}
Let $(Y,B)$ be a dlt pair and $f\colon (Y,B)\rightarrow T$ a log Calabi--Yau fibration.
Let \(U \coloneqq U_1 \cap U_2 \cap U_3\) be the intersection of the following dense open subsets of \(T\):
\begin{enumerate}[label=(\roman*)]
    \item the smooth locus $U_1 \subset T$, \label{item:nice-fibers-lc-CY-fibration-U_1}
    \item an open set $U_2$ over which the log fibers of $(Y,B)\rightarrow T$ are dlt and \(B_t = B|_{Y_t}\), and \label{item:nice-fibers-lc-CY-fibration-U_2}
    \item the complement $U_3$ of the support of the boundary divisor and the moduli divisor induced by the canonical bundle formula on $T$ by $(Y,B)$. \label{item:nice-fibers-lc-CY-fibration-U_3}
\end{enumerate}
Let \(E\) be an effective $\mathbb{Q}$-Cartier \(\bb Q\)-divisor \(E\) on $Y_U$ that does not contain any fiber of \(f\)
and does not contain any irreducible component of the restriction of 
any log canonical center of $(Y,B)$ to a fiber of $f$.
Then for \(\delta > 0\) small enough and
for any sequence
$Y_U \dashrightarrow W$ of steps of the $(K_{Y_U}+B_U+\delta E)$-MMP over $U$,
the normalizations of the log fibers of $(W,B_W)\rightarrow U$ are lc.
\end{lemma}

\begin{proof}
Let $d$ be the dimension of $U$, and let $t\in U$ be a closed point.
As in the proof of Proposition~\ref{lem:nice-fibers-klt-CY-fibration}, we may work on an open neighborhood of \(t\).
Using~\ref{item:nice-fibers-lc-CY-fibration-U_1},
let $(U,H_1+\dots+H_d;t)$ be a log smooth pair
with $t$ a log canonical center,
as in the proof of Lemma~\ref{lem:nice-fibers-klt-CY-fibration}.
By the canonical bundle formula
and~\ref{item:nice-fibers-lc-CY-fibration-U_3}, the pair 
\begin{equation}\label{eq:pair-pull-4}
(Y,B+f^*H_1+\dots+f^*H_d)
\end{equation}
is log canonical over a neighborhood of \(t\),
and $Y_t$ is a log canonical center.

Since \((Y,B)\) and \((Y_t, B_t)\) are dlt, the lc centers of \((Y_t, B_t)\) are irreducible components of
restrictions of lc centers of \((Y,B)\).
So, by assumption, $E$ contains no lc centers of $(Y_t,B_t)$.
Then inversion of adjunction \cite[Theorem 1]{Hacon14} and the assumption that \(Y_t\not\subset E\) imply that \(E\) contains no
log canonical centers of $(Y,B+f^*H_1+\dots+f^*H_d)$ near \(t\).
Thus, for sufficiently small \(\delta > 0\),
the pair
\((Y,B+f^*H_1+\dots+f^*H_d+\delta E)\)
is log canonical over a neighborhood of \(t\).

Let $Y_U\dashrightarrow W$ be a sequence
of steps of the $(K_{Y_U}+B_U+\delta E)$-MMP
over $U$,
and let $f'\colon W\rightarrow U$ be the induced fibration.
Let $B_W$ (resp. $E_W)$ 
be the pushforward of $B$ (resp. $E$) to $W$.
Then, over a neighborhood of \(t\), the pair 
\((W,B_W+{f'}^*H_1+\dots+{f'}^*H_d+\delta E_W)\)
is log canonical,
and hence so is
\begin{equation}\label{eq:pair-pull-5}
(W,B_W+{f'}^*H_1+\dots+{f'}^*H_d).
\end{equation} 
Since the log fiber $(W_t,B_t)$ is
the pair obtained by adjunction
of~\eqref{eq:pair-pull-5} to $W_t$,
we conclude by \cite[Theorem 1]{Hacon14} that
its normalization is lc.
\end{proof}

\section{Toric log Calabi--Yau pairs in families}\label{sec:toric}

In this section, we prove that in a family of log Calabi--Yau pairs,
the property of being toric is constructible. Later, in Section~\ref{sec:proofs-of-main-theorems}, we will use this result to prove Theorem~\ref{thm:toric-Fano}.

\begin{theorem}\label{thm:toric-CY}
Let $f\colon (\mathcal{X},\mathcal{B}) \rightarrow T$ be a family of log
Calabi--Yau pairs of index one, and assume that \(X_t\) is of Fano type for every closed point \(t\in T\). Then, the set 
\[
T_{\rm toric} \coloneqq\{ \text{closed }t \in T \mid (\mathcal{X}_t,\mathcal{B}_t) \text{ is a toric log CY pair}
\} 
\]
is a constructible subset of $T$. 
\end{theorem}

First, we prove that for certain nice fibrations,
the existence of a toric log fiber implies that every log fiber is toric. We will use the characterization of toric pairs proven by~\cite{BMSZ18}, which uses complexity (see Definition~\ref{defn:complexity}).

\begin{proposition}\label{prop:toric-log-CY-open} Let \(X\to T\) be a projective contraction of \(\bb Q\)-factorial normal quasi-projective varieties, 
\((X,B)\) a pair, and \(t_0\in T\) a closed point.
Assume the following conditions are satisfied:
\begin{enumerate}
\item \(\dim X_{t_0} = \dim X_t\) for \(t\) in \(T\), \label{item:condition-equidim}
\item \(B\) is reduced and contains no fibers of \(X \to T\), \label{item:condition-Q-Cartier-components}
\item
the log fibers of $(X,B)\rightarrow T$ are log canonical pairs, \label{item:condition-nearby-fibers-lc}
\item
the restriction of \(B\) to \(B|_{X_{t_0}}\) induces a bijection between prime components of $B$ and $B|_{X_{t_0}}$, \label{item:condition-restriction-bij-components}
\item
\(X \to T\) has injective class group restrictions (Definition~\ref{defn:class-group-restrictions}), and \label{item:D-num-trivial-in-nbhd}
\item \((X_{t_0}, B_{t_0})\) is a toric log CY pair. \label{item:condition-special-fiber-toric-lCY}
\end{enumerate}
Then all log fibers of \( (X,B) \to T\) are toric log CY pairs.
\end{proposition}

\begin{proof}
By assumption~\eqref{item:condition-special-fiber-toric-lCY}, the restriction of $K_X+B$ to $X_{t_0}$ is $\qq$-linearly trivial.
By~\eqref{item:D-num-trivial-in-nbhd}, we conclude that $K_X+B$ is $\qq$-linearly trivial on every closed fiber $X_t$.
In particular, by~\eqref{item:condition-Q-Cartier-components} and~\eqref{item:condition-nearby-fibers-lc}, the log fiber $(X_t,B_t)$ is a log CY pair for every $t\in T$.

Now write \(b\) for the number of prime components \(B_{t_0,1},\ldots,B_{t_0,b}\) of \(B_{t_0}\), and consider the decomposition \(\Sigma_{t_0} \coloneqq \sum_{i=1}^b B_{t_0, i}\) of \(B_{t_0}\). Since \((X_{t_0}, B_{t_0})\) is a toric pair, we have the equality \(b = \dim X_{t_0} + \rho(\Sigma_{t_0})\) (see, e.g., \cite[page 2]{BMSZ18}).
Now, for each \(1\leq i \leq b\), let \(\Gamma_i\) be the prime component of \(B\) restricting to \(B_{t_0, i}\) (using~\eqref{item:condition-restriction-bij-components}). Let \(t\in T\) be a closed point,
and consider the decomposition of \(B_t\) given by \(\Sigma_t \coloneqq \sum_{i=1}^b \Gamma_i|_{X_t} = B_t \).
Then \(|\Sigma_t| = b\), \(\dim X_t = \dim X_{t_0}\) by~\eqref{item:condition-equidim}, and \(\rho(\Sigma_t)  = \rho(\Sigma_{t_0})\) by ~\eqref{item:D-num-trivial-in-nbhd}, so  

\[
c(X_t, B_t; \Sigma_t) = c(X_{t_0}, B_{t_0}; \Sigma_{t_0}) = 0.
\] 

Hence \((X_t, B_t) = (X_t, \lfloor B_t \rfloor)\) is a toric log CY pair by Theorem~\ref{thm:BMSZ-toric}.
\end{proof}

\begin{theorem}\label{thm:toric-l-CY-open} Let \((X,B)\to T\) be a log Calabi--Yau fibration of index one over \(T\), i.e., $K_X+B\sim_T 0$. Assume \(X\to T\) is Fano type.
Then there is an open subset \(U\) of \(T\) such that if the log fiber \((X_t,B_t)\) is a toric CY pair for some closed \(t\in U\), then every log fiber over 
$U$ is a toric log CY pair.
\end{theorem}

\begin{proof}
    Let \(T'\to T\) be a finite Galois cover such that for the pullback \((X_{T'}, B_{T'})\), the prime components of \(B_{T'}\) restrict to prime divisors on general fibers of \(X_{T'}\to T'\). Let \((X'', B'')\to (X_{T'}, B_{T'})\) be a \(\bb Q\)-factorial dlt modification, which exists by \cite[Theorem 3.1]{KollarKovacs10}. We get a commutative diagram
    \[\xymatrix{(X,B) \ar[d] & (X_{T'}, B_{T'}) \ar[d] \ar[l] & (X'',B'')  \ar[l] \ar[ld]^-{\psi} \\ T & T' \ar[l]_-{\pi}}.\]
    Let \(U'\subset T'\) be a nonempty open subset over which \((X'',B'') \to U'\) satisfies conditions~\eqref{item:condition-equidim}, \eqref{item:condition-Q-Cartier-components}, and~\eqref{item:condition-restriction-bij-components} of Proposition~\ref{prop:toric-log-CY-open}. By shrinking \(U'\) if necessary, we may further assume that \((X'',B'') \to (X_{T'}, B_{T'})\) induces dlt modifications on log fibers over \(U'\).
    Lemma~\ref{lem:FT2} shows that \(X''|_{U'} \to U'\) is Fano type. We may assume by Lemma~\ref{lem:fano-type-num-triv-nbhd}\eqref{item:fano-type-num-triv-nbhd-part-2} that Proposition~\ref{prop:toric-log-CY-open}\eqref{item:D-num-trivial-in-nbhd} holds over \(U'\).
    Let \(U\subset T\) be an open subset whose preimage in \(T'\) is contained in \(U'\). We may assume \(\pi\) is \'etale over \(U\).
    Then \(\psi\colon (X'', B'') \to U'\) satisfies conditions \eqref{item:condition-equidim}--\eqref{item:D-num-trivial-in-nbhd} of Proposition~\ref{prop:toric-log-CY-open}.
    
    Now assume \((X_t, B_t)\) is toric log CY for some closed \(t\in U\), and let \(t' \in T'\) be a preimage of \(t\). Then \((X'_{t'}, B'_{t'})\) is toric log CY. A dlt modification of a toric log CY pair is also toric log CY by Lemma~\ref{lem:toric-mod}. Hence, the log CY pair \((X''_{t'}, B''_{t'})\) is toric.
    
    All log fibers of \(\psi\) over \(U'\) are toric log CY pairs by Proposition~\ref{prop:toric-log-CY-open}, so all fibers of \((X_{U'}, B_{U'}) \to U'\) are toric by Lemma~\ref{lem:toric-contr}. Hence, all log fibers of \((X,B) \to T\) over \(U\) are toric.
\end{proof}

\begin{proof}[Proof of Theorem~\ref{thm:toric-CY}]   
    We proceed by induction on \(\dim T\). If \(\dim T = 0\) there is nothing to show. Next, assume \(\dim T \geq 1\).
    As the geometric generic fiber is an index one log Calabi--Yau pair on a Fano type variety, we may take a dominant finite \'etale morphism \(T' \to T\) such that the base change \((\mathcal X_{T'}, \mathcal B_{T'}) \to T'\) is a log CY fibration and \(\mathcal X_{T'} \to T'\) is a Fano type morphism.
    Let \(U' \subset T'\) be the open subset of Theorem~\ref{thm:toric-l-CY-open}, and let \(U \subset T\) be its image.
    Then \(T'_{toric} \cap U'\) is a constructible subset of \(T'\) by Theorem~\ref{thm:toric-l-CY-open}, so its image \(T_{toric} \cap U\) in \(T\) is also constructible. Since \(T_{toric} \cap (T\setminus U)\) is constructible by induction, this completes the proof.
\end{proof}

We finish this section by proving a more general version of Theorem~\ref{thm:toric-Fano}.
Even though Theorem~\ref{thm:toric-CY} requires a Fano type assumption, we retain the theorem statement and the above proof because it will be useful in proving Theorem~\ref{thm:cluster-Fano}.
The proof of the following theorem was kindly suggested to us by J\'anos Koll\'ar and Burt Totaro. 

\begin{theorem}\label{thm:toric-constructible-any}
Let $\mathcal{X}\rightarrow T$ be a family of varieties (over $\mathbb{C}$).
Then, the set 
\[ 
T_{\rm toric}\coloneqq\{ \text{closed $t\in T$} \mid \mathcal{X}_t \text{ is a toric variety} \} 
\]
is a constructible subset of $T$. 
\end{theorem}

\begin{proof}
First, we assume that the closed fibers of $\mathcal{X}\rightarrow T$ are smooth varieties.
Then, there is a dense open subset \(U\subset T\) such that \(U\) is smooth and the morphism \(\mathcal X_U\to U\) is smooth, so Ehresmann’s theorem implies the closed fibers \(\mathcal X_t\) have constant Betti numbers over \(U\).
We may assume that the first Betti number of each fiber over \(U\) is \(0\), otherwise none of the fibers are toric (see, e.g., \cite[Theorem 12.1.10]{CLS11}).
By upper semicontinuity of \(h^0(\mathcal X, \mathcal T_{\mathcal X_t})\), we may restrict to a dense open subset \(V\subset U\) over which \(h^0(\mathcal X, \mathcal T_{\mathcal X_t})\) is constant.
Then \(\Aut^0(\mathcal X_V/V) \to V\) is a smooth family of algebraic groups.
Using the assumption that \(b_1(\mathcal X_t) = 0\),
we may apply
\cite[Th\'eor\`eme 1.7(a)]{SGA3XII} to conclude that the rank of \(\Aut^0(\mathcal X_t)\) is lower semicontinuous over \(V\).
Thus, we may further restrict to a dense open \(W\subset V\)
over which \(\Aut^0(\mathcal X_t)\) has constant rank.
Over \(W\), either all or none of the fibers are toric, since \(\mathcal X_t\) is toric if and only if the rank of \(\Aut^0(\mathcal X_t)\) is equal to \(\dim\mathcal X_t\). Then, Noetherian induction completes the proof in the case that the closed
fibers of $\mathcal{X}\rightarrow T$ are smooth varieties.

Now, we let $\mathcal{X}\rightarrow T$ be any family of varieties.
Let $\mathcal{Y}\rightarrow \mathcal{X}$ be a functorial resolution of
singularities (see, e.g.,~\cite[Theorem 3.36 and Theorem 3.35(5)]{Kol07}).
Then, for $t\in T$ general, we have an induced functorial resolution
$\mathcal{Y}_t\rightarrow \mathcal{X}_t$ (see, e.g.,~\cite[3.34.4]{Kol07}).
So there is a dense open subset \(U\) such that $\mathcal{Y}_t\rightarrow \mathcal{X}_t$ is a functorial
resolution of singularities for each closed point $t\in U$.
Therefore, the closed fibers of the morphism $\mathcal{Y}|_U\rightarrow U$ are smooth varieties.
By the first paragraph, we conclude that
\[
U_{\rm toric}(\mathcal{Y})\coloneqq\{ \text{closed $t\in U$} \mid \mathcal{Y}_t \text{ is a toric variety}\} 
\]
is a constructible subset of $U$.
Functorial resolutions are equivariant by~\cite[Proposition 3.9.1]{Kol07}; therefore, for \(t\in U\),
$\mathcal{Y}_t$ is toric whenever $\mathcal{X}_t$ is toric.
On the other hand, image of toric varieties via projective birational morphism are also toric by Lemma~\ref{lem:toric-contr}.
Therefore, for \(t\in U\), $\mathcal{X}_t$ is toric whenever $\mathcal{Y}_t$ is toric.
We conclude that $U_{\rm toric}(\mathcal{Y})=T_{\rm toric} \cap U$. 
Then, Noetherian induction completes the proof.
\end{proof} 

\section{Cluster type log Calabi--Yau pairs in families}\label{sec:cluster-type}

In this section, we prove that in a family of terminal log Calabi--Yau pairs,
the property
of being cluster type is constructible.
In Section~\ref{sec:proofs-of-main-theorems}, we will use this to prove Theorem~\ref{thm:cluster-Fano}.

\begin{theorem}\label{thm:CY-cluster}
Let $f\colon (\mathcal{X},\mathcal{B})\rightarrow T$ 
be a family of log Calabi--Yau pairs of index one,
and assume that $X_t$ is a $\qq$-factorial terminal Fano variety for every closed point $t\in T$.
Then, the set 
\[
T_{\rm cluster}\coloneqq
\{ \text{closed }t\in T \mid (\mathcal{X}_t,\mathcal{B}_t)
\text{ is a log CY pair of cluster type}\}
\]
is a constructible subset of $T$.
\end{theorem}

To prove Theorem~\ref{thm:CY-cluster}, we will use the characterization of cluster type pairs in Lemma~\ref{lem:dlt-mod-cluster},
which shows that \((X,B)\) is cluster type if and only if there is a \(\bb Q\)-factorial dlt modification \((Y, B_Y)\) that admits a crepant birational contraction \(\phi\colon (Y, B_Y) \dashrightarrow (\bb P^n, \Sigma^n)\).

Theorem~\ref{thm:CY-cluster} will follow from Theorem~\ref{thm:cluster-type-stabilize-open}. The idea of the proof of Theorem~\ref{thm:cluster-type-stabilize-open} is as follows.
Assume the log fiber over \(t_0\) is cluster type.
Then, after taking an appropriate \'etale cover of \(T\)
and suitable higher dlt modifications of \((Y_{t_0}, B_{Y,t_0})\), we may spread out \(\phi_{t_0}\) to a family of crepant birational contractions \(\phi_t \colon (Y_t, B_t) \dashrightarrow (W_t, B_{W, t})\) over an open neighborhood of a preimage of \(t_0\). Using Proposition~\ref{prop:toric-log-CY-open}, we show that the log CY pairs \((W_t, B_{W, t})\) are toric. Hence, the log fibers of \((\mathcal{X},\mathcal{B})\) are cluster type over an open neighborhood of \(t_0\).

In order to spread out the birational contraction \(\phi_{t_0}\), we will need to extend certain divisors over a family.
An important ingredient for this is the following lemma
on invariance of sections
for suitable dlt modifications of Fano type morphisms.
This lemma extends \cite[Lemma 2.7]{HX15}.

\begin{lemma}\label{lem:invariance-of-sections}
Let $X\rightarrow T$ be a Fano type morphism of \(\bb Q\)-factorial varieties.
Let $(X,B)$ be a
dlt pair that is log CY of index one over $T$.
Assume that for every closed point $t\in T$, each proper stratum of $(X,B)$ restricts to a proper stratum of $(X_t,B_t)$.
Assume there is a boundary $\Delta$ for which $(X,\Delta)$ is terminal
and log CY over $T$.
Then, there exists a non-empty affine open subset $U\subseteq T$
satisfying the following condition:
for any dlt modification $\pi \colon (Y,B_Y)\rightarrow (X,B)$,
any Cartier divisor $D$ on $Y$, and every closed point $t\in U$,
we have a surjective homomorphism
\[
H^0(Y,\mathcal{O}_Y(mD)) \twoheadrightarrow H^0(Y_t,\mathcal{O}_{Y_t}(mD|_{Y_t})),
\]
for every $m$ sufficiently divisible (independent of \(t\)).
\end{lemma}

\begin{proof}
Write $\Delta\sim_{T,\qq}\Delta'+A$
where $A$ is ample over $T$
and the pair $(X,\Delta')$ is terminal.
By shrinking \(T\), we may assume that no fibers of \(f\) are contained in \(\Delta\), \(\Delta'\), or \(A\).
Let $\psi\colon X'\rightarrow X$ be a log resolution
of $(X,{\rm supp}(\Delta)+{\rm supp}(B))$. By the assumption on the strata of $(X,B)$, after shrinking $T$, we may assume that \(\Pic T = 0\) and that the following hold for every closed point \(t\in T\):
\begin{enumerate}
    \item $X'_t\rightarrow X_t$
is a log resolution of
$(X_t, {\rm supp}(\Delta|_{X_t})+{\rm supp}(B|_{X_t}))$, and \label{item:invariance-of-sections-log-resol}
    \item every dlt modification
$(Y,B_Y)\rightarrow (X,B)$ restricts
to a dlt modification 
$(Y_t,B_{Y_t})\rightarrow (X_t,B_t)$. \label{item:invariance-of-sections-dlt}
\end{enumerate}

Now let \(\pi \colon (Y, B_Y) \to (X,B)\) be a dlt modification,
and let $D$ be a Cartier divisor on $Y$. We may write $D\sim_\qq \pi^*(D_X)+E_Y$
where $D_X$ is a $\qq$-Cartier divisor on $X$
and $E_Y$ is $\pi$-exceptional.
The log canonical places of \((X,B)\) are horizontal over \(T\) by~\eqref{item:invariance-of-sections-dlt}, so
after replacing $X'$ by a further blow-up along log canonical
places of $(X,B)$, we obtain a commutative diagram as follows, where
$\phi$ is a birational contraction and \(\psi\) still satisfies properties~\eqref{item:invariance-of-sections-log-resol} and~\eqref{item:invariance-of-sections-dlt} (without further shrinking \(T\)):
\[
\xymatrix{
(Y,B_Y) \ar[d]_-{\pi} &  \\ 
(X,B) \ar[d] & X' \ar[ld]\ar[l]^-{\psi}
 \ar@{-->}[ul]_-{\phi}\\ 
T. 
}
\]

Fix a rational number $\epsilon>0$ small enough so that all $\pi$-exceptional divisors are non-canonical places of $(X,\Gamma_\epsilon \coloneqq (1-\epsilon)B+\epsilon \Delta')$.
Write
\[
\psi^*(K_X+\Gamma_{\epsilon}) + F_\epsilon = K_{X'}+D'_\epsilon
\]
where $D'_\epsilon$ and $F_\epsilon$ are effective snc divisors 
without common components, and \(F_\epsilon\) is \(\psi\)-exceptional.
Since every $\pi$-exceptional divisor
is a non-canonical place of $(X,\Gamma_\epsilon)$,
we may assume that $F_\epsilon$ is $\phi$-exceptional.
Note that \((X', D'_\epsilon)\) is a klt pair by construction.

By Lemma~\ref{lem:div-combination},
every exceptional non-terminal place of $(X,\Gamma_\epsilon)$ is a log canonical place of $(X,B)$.
Therefore, every non-terminal center of $(X,\Gamma_\epsilon)$ is a stratum of 
$(X,\lceil D'_\epsilon \rceil$).
By the assumption on the strata of $(X,B)$, every proper stratum of $(X',\lceil D'_\epsilon \rceil)$ is horizontal over $T$ and restricts to a proper stratum on each closed fiber. 
Therefore, by successively performing blow-ups 
of the strata of $(X',\lceil D'_\epsilon \rceil)$, we may assume that $(X',D'_\epsilon)$ is a terminal pair (see, e.g.,~\cite[Proposition 2.36]{KM98}).

Let $\delta>0$ be a rational number small enough so that 
$A_{\epsilon,\delta} \coloneqq \epsilon A +\delta D_X$
is ample
and $(X', D'_{\epsilon,\delta} \coloneqq D'_\epsilon +\delta \phi^* E_Y)$ is terminal.
By~\cite[Theorem 1.8]{HMX13}, for every sufficiently divisible \(m\) and for every closed point \(t\in U\), the homomorphism
\begin{equation}\label{eq:surj-log-res}
H^0(X',\mathcal{O}_{X'}(m(K_{X'}+D'_{\epsilon,\delta}
+\psi^*A_{\epsilon,\delta}
)))
\rightarrow 
H^0(X'_t, 
\mathcal{O}_{X'_t}(m(K_{X'}+D'_{\epsilon,\delta}
+\psi^* A_{\epsilon,\delta})|_{X'_t}))
\end{equation} 
is surjective.
Since \(\phi\) is a birational contraction, we have \(\pi^* = \phi_* \circ \psi^*\), so
\begin{equation}\label{eq:pushforward}
\phi_*(m(K_{X'}+D'_{\epsilon,\delta}+\psi^* A_{\epsilon,\delta}))\sim_{\qq} \delta m D.
\end{equation}
Now let \(m\) be divisible enough, and let 
\[\Gamma_t \in H^0(Y_t,\mathcal{O}_{Y_t}(\delta m D|_{Y_t}))
 = H^0(Y_t, \mathcal O_{Y_t}(m(\pi^*(K_X+\Gamma_{\epsilon} + A_{\epsilon,\delta}) + \delta E_Y)|_{Y_t})) . 
\]
Pulling back $\Gamma_t$ via $\phi_t \colon X'_t \dashrightarrow Y_t$, we obtain
\[
\phi_t^*\Gamma_t \in H^0(X'_t, \mathcal{O}_{X'_t}( m (K_{X'}+D'_{\epsilon,\delta}+\psi^* A_{\epsilon,\delta}-F_\epsilon)|_{X'_t})),
\]
so we get
\[
\phi_t^*\Gamma_t +mF_\epsilon|_{X'_t} 
\in 
H^0(X'_t, \mathcal{O}_{X'_t}( m (K_{X'}+D'_{\epsilon,\delta}+\psi^* A_{\epsilon,\delta})|_{X'_t})).
\]
By the surjectivity of~\eqref{eq:surj-log-res}, the section 
$\phi_t^*\Gamma_t+mF_\epsilon|_{X'_t} $ is the restriction of a section
\[
\Gamma_{X'} \in H^0(X',\mathcal{O}_{X'}(m(K_{X'}+D'_{\epsilon,\delta}+\psi^* A_{\epsilon,\delta}))).
\]
By~\eqref{eq:pushforward}, we conclude
\(
\phi_*(\Gamma_{X'})\in H^0(Y,\mathcal{O}_Y(\delta mD))
\)
restricts to $\Gamma \in H^0(Y_t,\mathcal{O}_{Y_t}(\delta mD|_{Y_t}))$.
\end{proof}

\begin{definition}\label{defn:extension-prime-divisors-property}
{\em 
Let $f\colon X\rightarrow T$ be a fibration.
We say that $f$ satisfies the \defi{extension
of prime divisors property} if for every closed fiber $X_t$
and every effective prime divisor $P_t\subset X_t$,
there an effective divisor $P\subset X$ 
for which ${\rm supp}(P)|_{X_t}={\rm supp}(P_t)$.
}
\end{definition}

Next, we show that if a log CY fibration satisfies certain nice properties,
then after taking an \'etale cover of the base,
we can find a nice dlt modification with the property that any higher dlt modification has the extension of prime divisors property.
The proof of the following proposition uses a result of Totaro \cite[Theorem 4.1]{FHS21} on deformation invariance of the divisor class group for certain families of Fano varieties, after an \'etale cover of the base.

\begin{proposition}\label{prop:nice-FT-fibration}
Let $X$ be a terminal variety,
$f\colon X\rightarrow T$ a Fano fibration,
and $(X,B)$ an lc pair that is log CY of index one over $T$.
Assume that every fiber of $f$ is $\qq$-factorial.
After replacing \(T\) by a dense open subset,
there exists a commutative diagram:
\[
\xymatrix{
(X,B)\ar[d]_-{f} & (X_U,B_U)\ar[d]^-{f_U} \ar[l]_-{\phi} & (Z,B_Z) \ar[l] \ar[dl]^-{f_Z}  \\
T & U\ar[l] &  
}
\]
satisfying the following conditions:
\begin{enumerate}
\item $U\rightarrow T$ is a dominant finite \'etale morphism, the pair $(X_U,B_U)$ is obtained by base change, and $B_U$ contains no fibers of $f_U$, \label{item:nice-FT-fibration-etale-cover}
\item $(Z,B_Z)\rightarrow (X_U,B_U)$ is a $\qq$-factorial dlt modification that induces $\qq$-factorial dlt modifications on log fibers over closed points,
\label{item:nice-FT-fibration-phi-extraction}
\label{item:nice-FT-fibration-small-Q-fact}
\item the fibration $f_Z$ is of Fano type, \label{item:nice-FT-fibration-base-change-FT}
\item there is a boundary $\Delta_Z$ on $Z$ such that
$(Z,\Delta_Z)$ is terminal and log CY over $U$, \(\Delta_Z\) contains no fibers of \(f_Z\), and all log fibers of \((Z,\Delta_Z)\to U\) are terminal, \label{item:nice-FT-fibration-base-change-terminal-log-CY}
\item every stratum of $(Z,B_Z)$ restricts to a stratum on each closed log fiber of \(f_Z\), \label{item:nice-FT-fibration-dlt-centers}
\item every $\qq$-factorial dlt modification $\phi_Y\colon (Y,B_Y)\rightarrow (Z,B_Z)$
induces a \(\bb Q\)-factorial
dlt modification on every closed fiber of $(Z,B_Z)\rightarrow U$, and \label{item:nice-FT-fibration-dlt-modification}
\item for any $\qq$-factorial dlt modification $\phi_Y\colon (Y,B_Y)\rightarrow (Z,B_Z)$, the fibration $f_Z\circ \phi_Y$ has isomorphic class group restrictions and satisfies the extension of prime divisors property. \label{item:nice-FT-fibration-divisors-condition}
\end{enumerate}
\end{proposition}

\begin{proof}
As $X$ is terminal and is Fano over $T$,
then after possibly shrinking $T$, we can find a boundary $\Gamma$ such that $(X,\Gamma)$ is terminal and log Calabi--Yau over $T$.
After further shrinking $T$, we may assume that neither $\Gamma$ nor $B$ contains any fiber of $X\rightarrow T$, and that the fibers of \(f\) are Fano varieties with terminal singularities.

We first construct $U\rightarrow T$ and $(X_U,B_U)$ as in~\eqref{item:nice-FT-fibration-etale-cover}.
Since each fiber of $X\rightarrow T$ is $\qq$-factorial and terminal, it has rational singularities~\cite{Elkik81}
and is smooth in codimension two.
Moreover, by Kawamata--Viehweg vanishing, every fiber has acyclic structure sheaf,
so we may apply~\cite[Theorem 4.1]{FHS21} to find a dominant finite \'etale morphism $T'\rightarrow T$ from a smooth variety 
such that the base change $f'\colon X' \coloneqq X_{T'}\rightarrow T'$ has the property
that ${\rm Cl}(X'/T')$ maps surjectively to the class group of each closed fiber.
Let $(X',B')$ (resp. $(X',\Gamma')$) be 
the log pull-back of $(X,B)$ (resp. $(X,\Gamma)$) to $X'$.
Note that $f'$ is still a Fano type fibration,
$(X',B')$ is log Calabi--Yau of index one over $T'$,
and $(X',\Gamma')$ is terminal and log Calabi--Yau over $T'$.
By~\cite[Lemma 30]{KX16} applied to any dlt modification of $(X',B')$, we can find a dominant finite \'etale morphism
$U\rightarrow T'$ from a smooth variety
such that if $f_U\colon (X_U,B_U)\rightarrow T'''$ 
is the base change,
then the dual complex of $(X_U,B_U)$
restricts to the dual complex of the log fibers. 
In particular, each prime component of $B_U$ 
remains prime after restriction to any closed fiber.
We let $(X_U,\Gamma_U)$ be the log pullback of $(X',\Gamma')$ to $X_U$.
The morphism $f_U$ is Fano type, 
$(X_U,B_U)$ is log CY of index one over $U$,
and $(X_U,\Gamma_U)$ is terminal and log CY over $U$.
Further, \(\Cl(X_U/U)\) maps surjectively to the class group of each closed fiber of \(f_U\).

Now we turn to constructing the dlt modification
$(Z,B_Z)\rightarrow (X_U,B_U)$ in~\eqref{item:nice-FT-fibration-phi-extraction}.
By Lemma~\ref{lem:comb-terminal-dlt}, 
for sufficiently small $\epsilon>0$
we can find a $\qq$-factorial dlt modification
$(Z,B_Z)\rightarrow (X_U,B_U)$
such that
the log pullback $(Z,\Delta_Z)$
of $(X_U,(1-\epsilon)B_U+\epsilon\Gamma_U)$
is terminal.
We argue that the closed fibers of $Z\rightarrow U$ remain $\qq$-factorial.
For a closed point $t\in U$, the prime exceptional divisors of $Z_t\rightarrow X'_t$
are restrictions of the prime exceptional divisors of $Z\rightarrow X'$.
Indeed, this holds as the prime exceptional divisors of $Z\rightarrow X'$ are components of $B_Z$ and hence correspond to vertices of the dual complex $\mathcal{D}(Z,B_Z)$.
The prime exceptional divisors 
of $Z\rightarrow X'$ are $\qq$-Cartier.
Thus, the exceptional prime divisors
of $Z_t\rightarrow X'_t$ are $\qq$-Cartier, and so $Z_t$ is $\qq$-factorial.
We conclude that the $\qq$-factorial 
dlt modification $(Z,B_Z)\rightarrow (X_U,B_U)$ induces $\qq$-factorial dlt modifications on log fibers over closed points. This proves~\eqref{item:nice-FT-fibration-phi-extraction}.
Moreover, as the exceptional prime divisors over closed fibers are restrictions of the exceptional prime divisors of \(Z \to X_U\), the composition \(f_Z \colon Z \to U\) has surjective class group restrictions.
We have two pairs on the \(\bb Q\)-factorial variety \(Z\), namely:
\begin{enumerate}[label=(\roman*)]
    \item $(Z,B_Z)$, which is dlt and log CY of index one over $U$, and
    \item $(Z,\Delta_Z)$, which is terminal and log CY over $U$.
\end{enumerate}
The morphism $f_Z$ is Fano type by Lemma~\ref{lem:FT1}. 
After possibly shrinking \(U\) again, we may assume that \(U\) is smooth and affine,
all the log fibers of \((Z, \Delta_Z)\to U\) are terminal,
and \(f_Z\) has isomorphic class group restrictions (by Lemma~\ref{lem:fano-type-num-triv-nbhd}\eqref{item:fano-type-num-triv-nbhd-part-2}).
The following conditions are satisfied:

\begin{enumerate}[label=(\alph*)]
    \item the pair $(Z, B_Z)$ is $\qq$-factorial,
    dlt, and log CY of index one over $U$, \label{item:pf-of-nice-FT-fibration-B-lCY}
    \item the morphism $f_Z$ is of Fano type and has isomorphic class group restrictions, \label{item:pf-of-nice-FT-fibration-fano-type}
    \item the pair $(Z,\Delta_Z)$ is terminal
    and log CY over $U$, \label{item:pf-of-nice-FT-fibration-Delta-lCY}
    \item each stratum of $(Z,B_Z)$ restricts to a stratum of the log fiber over each closed point,
    \label{item:pf-of-nice-FT-fibration-lc-centers-restrict}
\end{enumerate}
showing that the assumptions of Lemma~\ref{lem:invariance-of-sections} hold, so after shrinking \(U\) again, we furthermore have
\begin{enumerate}[resume,label=(\alph*)]
  \item for every $t\in U$, every $\qq$-factorial dlt modification $\phi_Y\colon (Y,B_Y) \rightarrow (Z,B_Z)$,
    and every divisor $D$ on $Y$, we have a surjective homomorphism
\[
H^0(Y,\mathcal{O}_Y(mD)) \twoheadrightarrow H^0(Y_t,\mathcal{O}_{Y_t}(mD|_{Y_t})),
\]
for every $m$ sufficiently divisible. \label{item:pf-of-nice-FT-fibration-invariance-of-sections}
\end{enumerate}

Now we verify the conditions in the proposition statement. We've already seen that~\eqref{item:nice-FT-fibration-etale-cover} and \eqref{item:nice-FT-fibration-phi-extraction}.
Conditions~\ref{item:pf-of-nice-FT-fibration-fano-type} and~\ref{item:pf-of-nice-FT-fibration-Delta-lCY} imply~\eqref{item:nice-FT-fibration-base-change-FT} and~\eqref{item:nice-FT-fibration-base-change-terminal-log-CY}, respectively.
Conditions~\ref{item:pf-of-nice-FT-fibration-B-lCY} and~\ref{item:pf-of-nice-FT-fibration-lc-centers-restrict}
ensure that~\eqref{item:nice-FT-fibration-dlt-centers} and~\eqref{item:nice-FT-fibration-dlt-modification} hold.
Thus, it remains to show~\eqref{item:nice-FT-fibration-divisors-condition}.

First, we argue that $f_Z\circ \phi_Y$ has isomorphic class group restrictions.
For surjectivity, the class group of every fiber $Y_t$ is generated by the pullbacks of divisors from $Z_t$
and the prime exceptional divisors of $Y_t\rightarrow Z_t$.
By~\ref{item:pf-of-nice-FT-fibration-B-lCY} and~\ref{item:pf-of-nice-FT-fibration-lc-centers-restrict},
the prime exceptional divisors of $Y_t\rightarrow Z_t$ are restrictions of the prime exceptional divisors of $Y\rightarrow Z$.
Since \(Y\) is \(\bb Q\)-factorial, this shows \(Y_t\) is also \(\bb Q\)-factorial.
For injectivity,
if $D_t$ is a $\qq$-Cartier divisor on $Y$ that is $\qq$-linearly trivial on $Y_t$,
then ${\phi_{Y,t}}_*D_t$ is $\qq$-linearly trivial on $Z_t$.
The pushforward ${\phi_Y}_*D$ is $\qq$-linearly trivial over $U$ by~\ref{item:pf-of-nice-FT-fibration-fano-type},
so $D \sim_{U,\qq} E$ for some $\phi_Y$-exceptional divisor $E$.
Since $E_t$ is $\qq$-linearly trivial on  \(Y_t\) and exceptional over $Z_t$, the negativity lemma implies that $E_t$ is the trivial divisor.
Hence, since \(E\) is horizontal over the base by~\ref{item:pf-of-nice-FT-fibration-lc-centers-restrict}, we conclude that $E$ is trivial, so $D$ is $\qq$-linearly trivial over $U$.

It remains to show that $f_Z\circ \phi_Y$ satisfies the extension of prime divisors property.
Let $E_t\subset Y_t$ be an effective prime divisor.
Since $f_Z\circ \phi_Y$ has isomorphic class group restrictions,
there is a divisor $D$ on $Y$ for which $D|_{Y_t}\sim_{\bb Q} E_t$,
or, in other words, $0\leq mE_t \in H^0(Y_t,\mathcal{O}_{Y_t}(mD|_{Y_t}))$ for some \(m\).
By condition~\ref{item:pf-of-nice-FT-fibration-invariance-of-sections}, after replacing \(m\) by a sufficiently divisible multiple, there exists 
$0\leq F \in H^0(Y,\mathcal{O}_Y(mD))$ 
such that $F|_{Y_t}=mE_t$. 
In particular, $F$ is an effective divisor
for which ${\rm supp}(F)|_{Y_t}={\rm supp}(E_t)$.
\end{proof}

\begin{theorem}\label{thm:cluster-type-stabilize-open}
Let $X$ be a terminal variety, 
$f\colon X\rightarrow T$ a Fano fibration,
and $(X,B)$ an lc pair
that is log CY of index one over $T$.
Assume that every fiber of $f$ is $\qq$-factorial.
There is a non-empty open subset $V\subset T$ such that if $(X_t,B_t)$ is of cluster type for some $t\in V$, then there exists an open neighborhood $t\in V_0 \subseteq V$ over which every fiber is of cluster type.
\end{theorem}

\begin{proof}
Consider the following commutative diagram
provided by Proposition~\ref{prop:nice-FT-fibration}.
\[
\xymatrix{
(X,B)\ar[d]_-{f} & (X_U,B_U)\ar[d]^-{f_U} \ar[l]_-{\phi} & (Z,B_Z) \ar[l] \ar[dl]^-{f_Z}  \\
T & U\ar[l] &  
}
\]
By Proposition~\ref{prop:nice-FT-fibration}\eqref{item:nice-FT-fibration-base-change-FT} there exists a boundary $\Delta_Z$ such that $(Z,\Delta_Z)$ is terminal and log Calabi--Yau over $U$ and $\Delta_Z$ contains no fibers of $f_Z$.
Furthermore, $(Z,B_Z)\rightarrow (X_U,B_U)$ induces $\qq$-factorial dlt modifications on closed fibers by Proposition~\ref{prop:nice-FT-fibration}\eqref{item:nice-FT-fibration-dlt-modification}.
After possibly shrinking $U$, we may assume that the following conditions are satisfied:
\begin{enumerate}[label=(\alph*)]
    \item $U$ is a smooth affine variety, \label{item:4.5-open-U-is-smooth-affine}
    \item \(f_Z\) has equidimensional fibers, \label{item:4.5-f_Z-equidimensional}
    \item for every $\epsilon \in [0,1]$, the boundary divisor and the moduli divisor induced by the canonical bundle formula by $(Z,(1-\epsilon)B_Z+\epsilon\Delta_Z)$ on $U$ are trivial, and \label{item:4.5-boundary-moduli-divisor-trivial}
    \item no fiber of \(Z \to U\) is contained in \(B_Z\) or \(\Delta_Z\), every log fiber of $(Z,B_Z)\rightarrow U$ is dlt,
    and every log fiber of $(Z,\Delta_Z)\rightarrow U$ is terminal. 
    \label{item:4.5-log-fibers}
\end{enumerate}
Let $V \subset T$ be an open subset contained in the image of $U$;
we may further assume $V$ is affine.
We will show that \(V\) is the desired open subset in the proposition statement.

Thus, assume that $(X_{t_0},B_{t_0})$ is of cluster type 
for some $t_0\in V$,
and let $s_0\in U$ be a pre-image of $t_0$.
Then the dlt log CY pair $(Z_{s_0},B_{Z,s_0})$
is of cluster type by Lemma~\ref{lem:cluster-type-extraction} and Proposition~\ref{prop:nice-FT-fibration}.
By Lemma~\ref{lem:dlt-mod-cluster}, 
we have a commutative diagram
\begin{equation}\label{eq:comm-diag-central-fiber} 
\xymatrix{
(Y_{s_0},B_{Y,s_0})\ar[d]_-{\phi_{Y,s_0}}\ar@{-->}[rd]^-{\pi_{s_0}} & \\
(Z_{s_0},B_{Z,s_0}) \ar@{-->}[r] & (\pp^d,\Sigma^d),
}
\end{equation}
where $\phi_{Y,s_0}\colon (Y_{s_0},B_{Y,s_0})\rightarrow (Z_{s_0},B_{Z,s_0})$
is a \(\bb Q\)-factorial dlt modification
and 
$\pi_{s_0}\colon (Y_{s_0},B_{Y,s_0}) \dashrightarrow (\pp^d,\Sigma^d)$
is a crepant birational contraction.
Furthermore, if $E_{s_0} \coloneqq {\rm Ex}(\pi_{s_0})\setminus B_{Y,s_0}$,
then $(Y_{s_0},B_{Y,s_0}+\epsilon E_{s_0})$
is dlt for $\epsilon>0$ small enough.
By Lemma~\ref{lem:nice-dlt-mod-wrt-E},
we may pass to a higher dlt modification to assume that \(\phi_{Y,s_0}\) is
obtained by
blow-ups of strata of $(Z_{s_0},B_{Z,s_0})$.
The strata of $(Z_{s_0},B_{Z,s_0})$
are precisely the restrictions of the strata
of $(Z,B_Z)$ to the log fiber over $s_0$
by condition~\eqref{item:nice-FT-fibration-dlt-centers}.
Hence, there exists a \(\bb Q\)-factorial dlt modification $(Y,B_Y)\rightarrow (Z,B_Z)$,
obtained by a sequence of
blow-ups of strata of $(Z,B_Z)$,
such that the log fiber over $s_0$
of $(Y,B_Y)\rightarrow U$ is precisely
$(Y_{s_0},B_{Y,s_0})$.

We will show that, over a neighborhood \(U_0\) of \(s_0\), the log fibers of \((Y, B_Y)\) are cluster type CY pairs. To do this, we will construct a birational model \(Y \dashrightarrow W\) on which we may apply Proposition~\ref{prop:toric-log-CY-open}.
We show that $Y\dashrightarrow W$ induces
fiberwise birational contractions to toric log CY pairs. We construct \(W\) as follows.

By Proposition~\ref{prop:nice-FT-fibration}\eqref{item:nice-FT-fibration-divisors-condition}, there exists an effective divisor $E_Y$ on $Y$ such that ${\rm supp}(E|_{Y_{s_0}})={\rm supp}(E_{s_0})$. Note that \(E_Y\) has no common components with \(B_Y\).
By Lemma~\ref{lem:nice-dlt-mod-wrt-E}, after possibly replacing $(Y,B_Y)$ with a higher dlt modification,
we may assume that $(Y,B_Y+\delta E_Y)$ 
is dlt for $\delta>0$ small enough.
In particular, over an open neighborhood \(s_0 \in U_0 \subset U\), the effective divisor $E_Y$ contains no fiber of \(Y_{U_0} \to U_0\) and no intersection of a log canonical center of $(Y,B_Y)$ with a fiber of $Y_{U_0}\rightarrow U_0$.
By Lemma~\ref{lem:div-combination}, Lemma~\ref{lem:comb-terminal-dlt}, and Proposition~\ref{prop:nice-FT-fibration}\eqref{item:nice-FT-fibration-dlt-centers}, after possibly replacing  $(Y,B_Y)$
with a higher $\qq$-factorial dlt modification, we may assume that for \(\epsilon > 0\) small enough, the log pull-back 
$(Y,\Gamma_Y)$ of $(Z,(1-\epsilon)B_Z+\epsilon \Delta_Z)$ to $Y$
is terminal and the log fibers of $(Y,\Gamma_Y) \rightarrow U$ are terminal.
Then, for $\delta>0$ small enough, the pair
$(Y,\Gamma_Y+\delta E_Y)$ is terminal
and its restriction to closed fibers of $Y_{U_0}\rightarrow U_0$ are terminal.
By Lemma~\ref{lem:FT2}, the morphism  $f_Z\circ \phi_Y$ is Fano type.
Hence, we may run a $(K_Y+\Gamma_Y+\delta E_Y)$-MMP over $U_0$
that terminates with a good minimal model $W\rightarrow U_0$.

Let $B_W$ (resp. \(E_W\)) be the pushforward of $B_{Y, U_0}$ (resp. \(E_{Y, U_0}\)) to $W$.
Next, we will show that \((W, B_W)\) satisfies the hypotheses of Proposition~\ref{prop:toric-log-CY-open} over some open neighborhood of \(s_0\).

First, $W \to U_0 $ has injective class group restrictions by Proposition~\ref{prop:nice-FT-fibration}\eqref{item:nice-FT-fibration-divisors-condition} and Lemma~\ref{lem:injective-class-groups}.
Next, we apply Lemma~\ref{lem:nice-fibers-klt-CY-fibration} to $(Y,\Gamma_Y)$
and Lemma~\ref{lem:nice-fibers-lc-CY-fibration} to $(Y,B_Y)$ over \(U_0\).
The assumptions on Lemma~\ref{lem:nice-fibers-klt-CY-fibration} hold by~\ref{item:4.5-open-U-is-smooth-affine}, the choice of \(\epsilon\) above, and~\ref{item:4.5-boundary-moduli-divisor-trivial}.
The assumptions on Lemma~\ref{lem:nice-fibers-lc-CY-fibration} hold by~\ref{item:4.5-open-U-is-smooth-affine}, Proposition~\ref{prop:nice-FT-fibration}\eqref{item:nice-FT-fibration-dlt-modification}, and~\ref{item:4.5-boundary-moduli-divisor-trivial}.
By Lemma~\ref{lem:nice-fibers-klt-CY-fibration}\eqref{item:nice-fibers-klt-CY-fibration-normal-fibers-W} and~\eqref{item:nice-fibers-klt-CY-fibration-fibers-birational-contraction}, we conclude that $W_s$ is normal for every closed \(s\in U\),
and the corresponding birational maps $Y_s\dashrightarrow W_s$
are birational contractions. In particular \(Y_{U_0} \dashrightarrow W\) is a birational contraction; furthermore, by~\ref{item:4.5-f_Z-equidimensional}, \(W \to U_0\) has equidimensional fibers.
By Lemma~\ref{lem:nice-fibers-lc-CY-fibration}, the pair $(W_s,B_{W, s})$ is log canonical for every $s\in U_0$.

Next, since \(Y_{U_0} \dashrightarrow W\) is a birational contraction and contracts precisely \(E_{Y, U_0}\), the negativity lemma implies that \((W, B_W)\) is a relative log CY pair over $U$ and \((Y_{U_0}, B_{Y, U_0}) \dashrightarrow (W, B_W)\) is relative crepant birational over $U$. Furthermore, over each closed point \(s\in U_0\), the same argument and conclusions hold for the log fibers.
Over the point \(s_0\),
since \(E_{s_0}\) is \(\pi_{s_0}\)-exceptional,
the composition $(W_{s_0},B_{W,s_0})\dashrightarrow (Y_{s_0},B_{Y,s_0}) \dashrightarrow (\pp^d,\Sigma^d)$ is a crepant birational contraction that only contracts log canonical places of $(\pp^d,\Sigma^d)$.
So \((W_{s_0},B_{W,s_0})\) is a toric log CY pair by Lemma~\ref{lem:toric-mod}.
Thus, we have a \(\bb Q\)-factorial pair $(W,B_W)$ that is log CY
of index one and satisfies the following conditions:
\begin{enumerate}[label=(\roman*)]
    \item $\dim W_s = \dim W_{s_0}$ for every $s\in U_0$,
    \item the divisor $B_W$ is reduced and contains no fibers of \(W\to U_0\),
    \item the log fibers of $(W,B_W)\rightarrow U_0$ are log canonical pairs,
    \item the restriction of $B_W$ to $B_{W,s_0}$ induces a bijection between prime components, 
    \item the fibration $W \to U_0$ has injective class group restrictions, and
    \item $(W_{s_0},B_{W,s_0})$ is a toric log CY pair.
\end{enumerate}
By Proposition~\ref{prop:toric-log-CY-open}, all fibers of $(W,B_W)\rightarrow U_0$ are toric log CY pairs.

Let \(t_0 \in V_0 \subset V\) be an open subset whose preimage in \(U\) is contained in \(U_0\). We will show that the log fiber \((X_t, B_t)\) is a cluster type log CY pair for any closed point \(t\in V_0\).
If \(s\in U_0\) is a preimage of \(t\),
then we have a \(\bb Q\)-factorial dlt modification
\((Y_s,B_{Y,s})\to (Z_s, B_{Z,s}) \to (X_s, B_s)\)
by Proposition~\ref{prop:nice-FT-fibration}\eqref{item:nice-FT-fibration-small-Q-fact} and~\eqref{item:nice-FT-fibration-dlt-modification}.
By construction of \(U_0\), this admits a crepant birational contraction \((Y_s,B_{Y,s})\dashrightarrow (W_s,B_{W,s})\) to a toric log CY pair. Thus, Lemma~\ref{lem:dlt-mod-cluster} shows that \((X_s, B_s) \cong (X_t, B_t)\) is of cluster type.
\end{proof}

\begin{proof}[Proof of Theorem~\ref{thm:CY-cluster}]
We proceed by Noetherian induction. 
If $\dim T=0$, then there is nothing to prove, so assume $\dim T\geq 1$.
Take a dominant finite \'etale morphism \(T' \to T\) such that the base change \((\mathcal X_{T'}, \mathcal B_{T'})\) is a log CY pair of index one over \(T'\).
Let $V'\subseteq T'$ be the open subset 
provided by Theorem~\ref{thm:cluster-type-stabilize-open},
and let \(V \subseteq T\) be its image.
Then $T'_{\rm cluster}\cap V'$
is a constructible subset by Theorem~\ref{thm:cluster-type-stabilize-open}, and Noetherian induction,
so its image $T_{\rm cluster}\cap V$ in \(T\) is also constructible.
By Noetherian induction, we conclude that $T_{\rm cluster}$ is a constructible subset of $T$.
\end{proof}

\begin{remark}\label{rem:cluster-type-singularities-assumptions}
{\em 
    The terminal and \(\bb Q\)-factorial assumptions on the fibers are necessary for our proof of Theorem~\ref{thm:CY-cluster}, in particular for our method of spreading out the crepant birational contraction from \((Y_{s_0}, B_{Y, s_0})\) to a toric log CY pair in the proof of Theorem~\ref{thm:cluster-type-stabilize-open}.
We use the terminality assumption to apply Lemma~\ref{lem:div-combination}. We use the terminality and \(\bb Q\)-factoriality assumptions to apply \cite[Theorem 4.1]{FHS21} in the proof of Proposition~\ref{prop:nice-FT-fibration}.
}
\end{remark}

\section{Proofs of main theorems}\label{sec:proofs-of-main-theorems}
Now we turn to proving the two main theorems of this article, using Theorems~\ref{thm:toric-CY} and~\ref{thm:CY-cluster}. We will need the following lemmas.

\begin{lemma}\label{lem:index-constr}
Let $(\mathcal{X},\mathcal{B})\rightarrow T$ be a family of log CY pairs.
Then there is a locally closed stratification of $T$ such that the log fibers over each stratum have the same index.
\end{lemma}

\begin{proof}
By Noetherian induction, it suffices to show that there exists a non-empty open set of $T$ over which all log fibers have the same index.
As the geometric generic log fiber is log CY, we may find a dominant finite \'etale morphism
$T'\rightarrow T$ such that
${\rm Pic}(T')=0$ and
the base change
$(\mathcal{X}_{T'},\mathcal{B}_{T'}) \to T'$ is a log CY fibration.
Thus, over a non-empty open subset $U' \subset T'$, we may find the smallest positive integer $m$ for which $m(K_{\mathcal{X}_{T'}}+\mathcal{B}_{T'})|_{\mathcal X_{U'}} = m(K_{\mathcal{X}_{U'}}+\mathcal{B}_{U'})\sim 0$.
In particular, the index of each log fiber of $(\mathcal{X}_{U'},\mathcal{B}_{U'})\rightarrow {U'}$ divides $m$.
Let $\mathcal{Y}_{U'} \rightarrow \mathcal{X}_{U'}$ be the index one cover of $K_{\mathcal{X}_{U'}}+\mathcal{B}_{U'}$.
By construction, the log fibers of $(\mathcal{X}_{U'},\mathcal{B}_{U'})\rightarrow {U'}$ with index $m$ are precisely the connected fibers of $\mathcal{Y}\rightarrow {U'}$.
Hence, there is an open subset of ${U'}$ over which all log fibers of $(\mathcal{X}_{U'},\mathcal{B}_{U'})\rightarrow {U'}$ is exactly $m$.
\end{proof}

\begin{lemma}\label{lem:log-canonical-constr}
Let $(\mathcal{X},\mathcal{B})\rightarrow T$ be a family of pairs.
Then the subset of $T$ parametrizing log fibers with log canonical singularities is constructible. 
\end{lemma}

\begin{proof}
Let \(T'\to T\) be a dominant finite \'etale morphism such that the base change \((\mathcal X_{T'}, \mathcal B_{T'})\) admits a log resolution \(\mathcal X' \to \mathcal X_{T'}\) over \(T'\) and such that \(\mathcal B_{T'}\) does not contain any fibers of \(\mathcal X_{T'} \to T'\).
By Noetherian induction, it suffices to show that the condition holds over some non-empty open subset of the base.
There is a non-empty open subset $U\subset T'$ over which all the exceptional divisors of $\mathcal{X}'$ are horizontal over $U$.
By possibly shrinking $U$, we may assume that $\mathcal{X}'\rightarrow \mathcal{X}_{T'}$ induces a log resolution on log fibers of $(\mathcal{X}_{T'},\mathcal{B}_{T'})\rightarrow T$ over closed points of $U$.
Therefore, over $U$, either all or none of the log fibers of $(\mathcal{X},\mathcal{B})\rightarrow T$ are log canonical.
This finishes the proof.
\end{proof}

\begin{proof}[Proof of Theorem~\ref{thm:toric-Fano}]
By Noetherian induction,
it suffices to show that $T_{\rm toric} \cap U$ is constructible for some open dense affine subset $U\subseteq T$.
By \cite[Lemma 2.7]{HX15}, there is a dense open subset \(U\subset T\) such that for $m$ divisible enough, the map
\(
H^0(\mathcal{X},-mK_{\mathcal{X}})\rightarrow
H^0(\mathcal{X}_t,-mK_{\mathcal{X}_t})
\)
is surjective (and nonzero) for every closed point $t\in U$.
Fix such an \(m\),
let 
\[
\mathcal B \subset \mathcal X_U\times |-m K_{\mathcal X_U}|
\]
be the universal divisor, and consider the family \(\pi\colon (\mathcal X_U \times |-m K_{\mathcal X_U}|, \frac{1}{m} \mathcal B) \to U \times |-m K_{\mathcal X_U}|\).
The fiber of \(\pi\) over $(t,p)$ is \((\mathcal X_t, \frac{1}{m} B_p|_{\mathcal X_t})\), where \(\mathcal B_p \in |-m K_{\mathcal X}|\) is the effective divisor corresponding to \(p\).
Let $W\subseteq U \times |-m K_{\mathcal X_U}|$ be the open subset consisting of points $(t,p)$ for which $B_p$ does not contain the fiber $\mathcal{X}_t$, and write \((\mathcal{Y}_W,\mathcal{B}_W) \to W\) for the restriction of the family \(\pi\) to \(W\).

Then \((\mathcal{Y}_W,\mathcal{B}_W) \to W\) is a family of pairs.
Let \(W_{\rm toric} \subset W\) be the locus parametrizing the log fibers that are toric log CY pairs.
Since a variety \(X\) is toric if and only if it admits a boundary \(B\) for which \((X,B)\) is a toric log CY pair,
the intersection \(T_{\rm toric} \cap U\) is equal to the image of \(W_{\rm toric}\) under the projection \(W \subset U \times |-mK_{\mathcal X_U}| \xrightarrow{\pi_1} U\). Thus, it suffices to show that \(W_{\rm toric} \subset W\) is a constructible subset.

By Lemma~\ref{lem:index-one}, \(W_{\rm toric}\) is contained in the set \(W_1\) parametrizing log canonical log fibers of index one, and \(W_1 \subset W\) is constructible
by Lemmas~\ref{lem:index-constr} and~\ref{lem:log-canonical-constr}.
Over the normalization of each irreducible component of $W_1$, the pair
$(\mathcal{Y}_W,\mathcal{B}_W)$ induces a family of log CY pairs of index one for which every closed fiber $\mathcal{Y}_t$ is Fano.
By Theorem~\ref{thm:toric-CY}, we conclude that the subset 
\( W_{\rm toric} \subseteq W_1 \subseteq W\)
parametrizing toric log CY fibers is constructible,
as it is a constructible subset of a constructible subset. This completes the proof.
\end{proof}

\begin{proof}[Proof of Theorem~\ref{thm:cluster-Fano}]
The proof is identical to the proof of Theorem~\ref{thm:toric-Fano}, using Theorem~\ref{thm:CY-cluster} instead of Theorem~\ref{thm:toric-CY}.
\end{proof}

\begin{proof}[Proof of Corollary~\ref{introcor:ct-smooth-Fano}]
The \(n\)-dimensional smooth Fano varieties form a bounded family by~\cite[Theorem 0.2]{KMM92}; that is, there exist finitely many surjective projective morphisms \(g_i\colon \mathcal Z_i \to S_i\) of quasi-projective varieties such that any \(n\)-dimensional smooth Fano variety is isomorphic to a closed fiber of \(g_i\) for some \(i\) (see, e.g., \cite[Section 2.8]{Birkar21}).
Furthermore, by standard arguments, one may replace each \(S_i\) by a constructible subset to assume every closed fiber is an \(n\)-dimensional smooth Fano variety, i.e., that the \(g_i\) are parametrizing families for \(n\)-dimensional smooth Fano varieties. Then the corollary is immediate from Theorem~\ref{thm:cluster-Fano}.
\end{proof}

We finish this section, by proving that the analogous statement to Corollary~\ref{introcor:ct-smooth-Fano} fails if we replace cluster type with rational.
The following example was kindly suggested to us by Brendan Hassett.

\begin{theorem}[Theorem~\ref{thm:rat-fam}]
    Let \(n\geq 4\) be an integer. Let \(f_i\colon \mathcal X_i \to T_i\) a finite set of smooth projective morphisms whose fibers are smooth rational Fano varieties of dimension \(n\). Then, there exists an \(n\)-dimensional smooth rational Fano variety that is not isomorphic to a fiber of any \(f_i\). That is, there do not exist finitely many families parametrizing \(n\)-dimensional smooth rational Fano varieties.
\end{theorem}

\begin{proof}
    Let $B$ be the open subset of $\pp^{59}$ parametrizing all smooth $(2,2)$ hypersurfaces in $\pp^2\times \pp^3$. Let $U\rightarrow B$ be the universal family, and let \(\pi\colon \mathbb P^{n-4}\times U \to B\) be the composition with the projection to \(U\). By \cite{HassettPirutkaTschinkel}, we know that the set $B_{\rm rat}$ parametrizing rational fibers of \(\pi\) is a countable union of proper closed subset of \(B\), and \(B_{\rm rat}\subset B\) is Zariski (and even Euclidean) dense.
    We argue that there is a closed point $b\in B_{\rm rat}$ for which $\mathbb P^{n-4}\times U_b$ is not isomorphic to a fiber of any $f_i$.

    Let \(\mathcal H_{\mathbb P^{11}}\) be the Hilbert scheme parametrizing closed subschemes of \(\mathbb P^{11}\) with the same Hilbert polynomial as a \((2,2)\) hypersurface in \(\mathbb P^2\times\mathbb P^3\) under the Segre embedding, and let \(\mathcal U_{\mathbb P^{11}} \to \mathcal H_{\mathbb P^{11}}\) be the universal family.
    Let \(\mathcal H_{\mathbb P^2\times\mathbb P^3}\subset \mathcal H_{\mathbb P^{11}}\) be the subset parametrizing such subschemes that are contained in \(\mathbb P^2\times\mathbb P^3\). Then \(\mathcal H_{\mathbb P^2\times\mathbb P^3}\) parametrizes \((2,2)\) hypersurfaces in \(\mathbb P^2\times\mathbb P^3\), and by construction, \(\mathcal H_{\mathbb P^2\times\mathbb P^3}\subset \mathcal H_{\mathbb P^{11}}\) is a locally closed subset.
    The family \(\mathbb P^{n-4}\times\mathcal U_{\mathbb P^{11}} \to \mathcal H_{\mathbb P^{11}}\) induces a morphism \(\mathcal H_{\mathbb P^{11}} \to \mathcal H_{\mathbb P^{12 (n-4) + 11}}\) to the corresponding component of the Hilbert scheme of \(\mathbb P^{12(n-4)+11}\). Let \(\mathcal H_0 \subset \mathcal H_{\mathbb P^{12 (n-4) + 11}}\) denote the image of \(\mathcal H_{\mathbb P^2\times\mathbb P^3}\); this is a constructible subset.
    Let \(\phi_B\colon B\to \mathcal H_{\mathbb P^2\times\mathbb P^3}\) be the induced morphism to the Hilbert scheme. The image \(\phi_B(B_{\rm rat}) \subset \phi_B(B)\) is a proper subset that is Zariski dense. Furthermore, if \(\psi_B\colon B \to \mathcal H_{\mathbb P^2\times\mathbb P^3}\to\mathcal H_0\) is the composition, then \(\psi_B(B_{\rm rat}) \subset \psi_B(B) = \mathcal H_0\) is again a Zariski dense proper subset.

    Now we consider the families \(f_i\colon\mathcal X_i\to T_i\). Since embedding dimension is upper semicontinuous in families, we may replace the \(T_i\) by a (Zariski) locally closed stratification so that for each \(i\), the fibers of \(f_i\) all have embedding dimension \(\leq 12(n-4)+11\) or all have strictly higher embedding dimension. Let \(i\) be an index for which this embedding dimension is \(\leq 12(n-4)+11\). Then we get an induced morphism \(\psi_i\colon T_i \to \mathcal H_{\mathbb P^{12 (n-4) + 11}}\) to the Hilbert scheme, and the image \(\psi_i(T_i)\) is constructible. Then the intersection \(\psi_i(T_i)\cap\mathcal H_0\) is a constructible subset of \(\mathcal H_0\), so it is a finite union of locally closed subsets \(W_{ij}\) of \(\mathcal H_0\).
    Each \(W_{ij}\) must be contained in a proper Zariski closed subset of \(\mathcal H_0\); otherwise, \(W_{ij}\) would contain a nonempty open subset of \(\mathcal H_0\) and hence would intersect the complement of \(\psi_B(B_{\rm rat})\), implying that \(f_i\) has an irrational fiber. Therefore, the intersection \(\psi_i(T_i)\cap\mathcal H_0\) is not dense in \(\mathcal H_0\). We conclude that there exists a closed point \(b\in\psi_B(B_{\rm rat})\) whose image in \(\mathcal H_{\mathbb P^{12 (n-4) + 11}}\) is not contained in \(\bigcup_i \psi_i(T_i)\); that is, no fiber of any \(f_i\) is isomorphic to the rational variety \(\mathbb P^{n-4}\times U_b\).
\end{proof}

\section{Examples and questions}\label{sec:examples-and-questions}

In this section, we collect some examples
of the behavior of toricity and cluster type
in families of Fano varieties.
We propose some questions for further research.
The first two examples show that the property of being a toric variety is neither open nor closed. 

\begin{example}\label{ex:toric-not-closed}
{\em 
Many smooth Fano varieties admit singular toric degenerations.
For instance, the Mukai--Umemura threefold \(V_{22}\) admits a degeneration to a Gorenstein terminal toric Fano $3$-fold (see, e.g.,~\cite[Theorem 1.7]{Gal18}). In this degeneration, all the fibers except the central fiber are non-toric.
}
\end{example}

\begin{example}[\cite{HKW24}]\label{ex:toric-not-open}
{\em
The projective plane $\mathbb{P}^2$ admits non-toric, quasismooth degenerations.
Namely, let $(a,b,c)$ and $(a,b,d)$ be two pairs of Markov triples for which 
$d=3ab-c$, the so-called {\em adjacent} Markov triples.
Consider the weighted hypersurface
\[
V(x_1x_2+x_3^{c}+x_4^{d})\subset \pp(a^2,b^2,d,c) .
\]
By~\cite{HKW24}, this 
is a non-toric, quasismooth, rational Fano $\mathbb{G}_m$-surface of
Picard rank one,
which is a degeneration of $\pp^2$.
Furthermore, by \cite[Theorems 1.2 and 1.3]{HKW24}, all the non-toric,
normal, rational projective
$\mathbb{G}_m$-surfaces which are degenerations of $\pp^2$ arise in this way.
Hence, the property of being a toric variety is not closed in families of Fano varieties.
}
\end{example}

The following example shows that the property of being cluster type
is not open in families of Fano varieties.
We give a family of canonical Fano threefolds such that the special fiber is of cluster type but not toric, and nearby fibers are irrational and hence not of cluster type.

\begin{example}\label{ex:cluster-not-closed}
{\em 
Let \(X \subset \bb P^4\) be the singular cubic threefold defined by
\[ f(x_0,\dots,x_4) \coloneqq (x_0 + x_1 + x_2 + x_3 + x_4)^3 - x_0^3 - x_1^3 - x_2^3 - x_3^3 - x_4^3. \]
Let \(g(x_0,\dots,x_4)\) be the equation of a general cubic, and
consider the family of Fano threefolds
\[
\mathcal{X}\coloneqq\{ ([x_0:\dots:x_4],t) \mid  (1-t)f(x_0,\dots,x_4)+tg(x_0,\dots,x_4)=0\}
\subset \pp^4 \times \mathbb{A}^1
\] 
over \(\mathbb{A}^1\).
The general fiber of $\mathcal{X} \to \mathbb{A}^1$ is a smooth cubic threefold, so it is irrational.

Now we show that the special fiber $\mathcal{X}_0 = X$ is of cluster type but not toric. Since \(X\) is isomorphic to the Segre cubic
\[
\{ [x_0:\dots:x_5] \mid x_0+\dots+x_5 = x_0^3+\dots+x_5^3=0\}\subset \mathbb{P}^5,
\]
it has $10$ nodal singularities and
its automorphism group is $S_6$; in particular, \(X\) is not a toric variety.
It remains to construct a boundary \(B\) on \(X\) such that \((X,B)\) is a log CY pair of cluster type.
There is a small resolution $\pi\colon Y\rightarrow X$
which is isomorphic to the blow-up $\phi\colon Y \to \pp^3$
at $5$ general points
$\{p,q,r,s,t\} \subset \bb P^3$~\cite{Kapranov93}.
Let $H_1$ (resp. $H_2$) by the hyperplane in $\pp^3$
spanned by $\{p,q,r\}$ (resp. $\{p,s,t\}$).
Since \(\{p,q,r,s,t\}\) are in general position, the line $\ell_1 \coloneqq H_1 \cap H_2$ only contains one point of the set $\{p,q,r,s,t\}$, namely $p$. 
Next, let $u_1$ and $u_2$ be two general points in $\ell_1$,
and let $H_3$ (resp. $H_4$) be the hyperplane in $\pp^3$ spanned by $\{u_1,q,s\}$
(resp. $\{u_2,r,t\}$).
By construction, the pair
$(\pp^3,H_1+H_2+H_3+H_4)$ is log Calabi--Yau,
and each point of the set $\{p,q,r,s,t\}$ is contained in precisely two of the hyperplanes
$\{H_1,H_2,H_3,H_4\}$. 
In particular, if we write 
\[
\phi^*(K_{\pp^3}+H_1+H_2+H_3+H_4)=K_Y+B_Y,
\]
then the divisor $B_Y$ is effective. 
Then $(X,B \coloneqq \pi_* B_Y)$ is a log CY pair,
and the exceptional set of 
the crepant birational map
\[
\pi\circ \phi^{-1} \colon (\pp^3,H_1+H_2+H_3+H_4) \dashrightarrow (X,B)
\]
is precisely the $10$ lines passing through pairs of points in the set $\{p,q,r,s,t\}$.
In particular, the intersection $\mathbb{G}_m^3\cap {\rm Ex}(\pi\circ \phi^{-1})$ contains no divisors, 
so the pair $(X,B)$ is of cluster type.
This shows that the Fano variety $X$ is of cluster type.
}
\end{example}

The following example shows that the property of being cluster type is not closed in families of Gorenstein canonical Fano surfaces. 

\begin{example}\label{ex:cluster-not-open}
{\em  
Let $X_0$ be the Gorenstein del Pezzo surface of Picard rank one with a single $D_5$ singularity.
By~\cite[Theorem 1.1 and Example 8.3]{HP10}, we know that $X_0$ admits a $\qq$-Gorenstein smoothing $\mathcal{X}\rightarrow \mathbb{A}^1$.
Since the anticanonical volume of $X_0$ is $4$ (see, e.g.,~\cite[Lemma 3]{MZ88}),
the general fiber of $\mathcal{X}\rightarrow \mathbb{A}^1$ is a smooth del Pezzo surface of degree $4$.
Every smooth del Pezzo surface of degree $4$
admits a $1$-complement of coregularity zero~\cite[Theorem 2.1]{ALP23}
and so it is of cluster type by~\cite[Lemma 1.13]{GHK15}. 
On the other hand, the central fiber $\mathcal{X}_0$ is not of cluster type, 
as cluster type surfaces only admit $A$-type singularities (see~\cite[Proposition 5.1]{EFM24}).
}
\end{example}

The following example shows that the property of being a (possibly non-normal) toric variety is not locally closed in families of slc varieties.

\begin{example}\label{ex:toric-not-locally-closed}
{\em 
Consider the trivial family $\pi_2\colon \mathcal{Y}\coloneqq\pp^2\times \mathbb{A}^1_t\rightarrow \mathbb{A}^1_t$.
Now, let $\mathcal{D}$ be a divisor on $\mathcal{Y}$ for which $\mathcal{D}_0$ is the reduced sum of the coordinate hyperplanes
and $\mathcal{D}_t$ is a smooth cubic curve for every $t\neq 0$. 
Let $\mathcal{X}$ be the deformation of $\mathcal{Y}$ to the normal cone of $\mathcal{D}$.
For each $t\in \mathbb{A}^1$, we let $\psi_t \colon \mathcal{X}_t \rightarrow \mathbb{A}^1_s$
be the induced deformation of $\mathcal{Y}_t$ to the normal cone over $\mathcal{D}_t$.
Then, the threefold $\mathcal{X}$ admits a morphism
$\psi\colon \mathcal{X}\rightarrow \mathbb{A}^2$ whose fiber over $(s,t)$ is $\psi_t^{-1}(s)$.
By construction, all the fibers of $\psi$ over $\mathbb{G}_m\times \mathbb{A}^1_t$ are isomorphic to $\pp^2$ and hence toric varieties.
The fibers over $\{0\}\times \mathbb{G}_m$ are cones over smooth cubic curves and so they are not toric. 
The fiber over $(0,0)$ is the non-normal cone over the reduced sum of the coordinate hyperplanes of $\pp^2$, which is a toric variety.
We conclude that the subset of $\mathbb{A}^2$
parametrizing toric fibers of $\psi\colon \mathcal{X}\rightarrow \mathbb{A}^2$ is
\[
(\mathbb{G}_m\times \mathbb{A}^1) \cup (0,0)
\]
which is not a locally closed subset. However, it is a constructible subset.
}
\end{example}

The previous examples motivate the following question.

\begin{question}
    Is log rationality a constructible condition in families of Fano varieties?
\end{question}

Rationality specializes in smooth families~\cite{KontsevichTschinkel}. 
Thus, the following question seems to be reasonable
to study log rationality in families.

\begin{question}
 Let $(\mathcal{X},\mathcal{B})\rightarrow C$ be a log smooth family for which the generic fiber is log rational. Are closed fibers log rational?
\end{question}

The previous question is closely related to the log rationality conjecture proposed
by Ducat and Enwright, Figueroa, and the second author (see~\cite[Conjecture 1.4]{Duc22} and~\cite[Conjecture 1.1]{EFM24}).

Finally, our proof of Theorem~\ref{thm:cluster-Fano} requires that the fibers of the family  are \(\bb Q\)-factorial and terminal (see Remark~\ref{rem:cluster-type-singularities-assumptions}). However, we are not aware of any counterexamples to Theorem~\ref{thm:cluster-Fano} when these assumptions on the singularities are removed.
\begin{question}
    Is the property of being cluster type a constructible condition in families of Fano varieties with worse than terminal \(\bb Q\)-factorial singularities?
\end{question}

\bibliographystyle{habbvr}
\bibliography{references}

\end{document}